\documentclass[reqno,11pt,twoside]{article}

\usepackage[hmargin=1in,vmargin=1in, a4paper, centering]{geometry}
\usepackage{appendix}

\usepackage{amssymb, amsmath, amsthm}
\usepackage{fourier}
\usepackage{ebgaramond}
\usepackage{esint}
\usepackage{tikz-cd}
\usepackage{graphicx}
\usepackage{mathtools}
\mathtoolsset{showonlyrefs}
\usepackage{xspace}
\theoremstyle{definition}
\newtheorem{defi}{Definition}[section]
\newtheorem{example}[defi]{Example}
\newtheorem{remark}[defi]{Remark}

\theoremstyle{plain}
\newtheorem{theorem}[defi]{Theorem}
 \newtheorem{prop}[defi]{Proposition}
\newtheorem{lemma}[defi]{Lemma}
\newtheorem{cor}[defi]{Corollary}

\newtheorem{porism}[defi]{Porism}

\usepackage{marginnote}

\usepackage{color}

\numberwithin{equation}{section}

\usepackage[pdftitle={Discrepancy for cut and project sets},
  pdfauthor={Henna Koivusalo and Jean Lagacé}, 
pdfborder={0 0 0}
]{hyperref}

\makeatletter

\let\TagsLeftOn\tagsleft@true
\let\TagsLeftOff\tagsleft@false
\makeatother

\newcommand{\B}{\mathbb B}
\newcommand{\D}{\mathbb{D}}
\newcommand{\R}{\mathbb R}
\newcommand{\N}{\mathbb N}
\newcommand{\Z}{\mathbb Z}
\newcommand{\C}{\mathbb{C}}

\newcommand{\E}{\mathbb {E}}

\newcommand{\F}{\mathbb F}

\newcommand{\de}{\, \mathrm{d}}
\newcommand{\del}{\partial}

\newcommand{\Vol}{\operatorname{Vol}}

\newcommand{\CA}{\mathcal{A}}

\newcommand{\CF}{\mathcal{F}}

\newcommand{\CS}{\mathcal{S}}

\newcommand{\CX}{\mathcal{X}}

\newcommand {\bx}{\mathbf x}

\newcommand {\be}{\mathbf e}

\newcommand {\bk}{\mathbf k}

\newcommand {\by}{\mathbf y}

\newcommand {\bs}{\mathbf z}
\newcommand {\bu}{\mathbf u}

\newcommand {\bn}{\mathbf n}

\newcommand {\bxi}{\boldsymbol \xi}
\newcommand{\fd}{s}

\newcommand {\RL}{\mathrm L}

\newcommand{\RM}{\mathrm M}

\newcommand{\RC}{\mathrm C}

\newcommand{\dwn}{\triangledown}
\newcommand{\up}{\triangle}
\newcommand{\lef}{\triangleleft}

\newcommand{\abs}[1]{\left\lvert #1 \right\rvert}

\newcommand{\set}[1]{\left\{ #1 \right\}}

\newcommand{\norm}[1]{\left\| #1 \right\|}
\newcommand{\bo}\boldsymbol{}
\newcommand{\bigo}[2][]{O_{#1}\left( #2 \right)}
\newcommand{\smallo}[2][]{o_{#1}\left( #2 \right)}

\newcommand{\qand}{\quad \text{and}\quad}

\DeclareMathOperator{\inte}{int}
\DeclareMathOperator{\dist}{dist}

\DeclareMathOperator{\spn}{span}

\DeclareMathOperator{\diam}{diam}

\DeclareMathOperator{\vol}{vol}
\DeclareMathOperator{\covol}{covol}

\DeclareMathOperator{\grad}{grad}
\renewcommand{\le}{\leqslant}
\renewcommand{\ge}{\geqslant}

\renewcommand{\tilde}{\widetilde}
\newcommand{\eps}{\varepsilon}
\renewcommand{\phi}{\varphi}

\DeclareMathOperator{\supp}{supp}

\renewcommand{\mod}[1]{\,({\rm mod}\,#1)}

\DeclareMathOperator{\GL}{GL}

\renewcommand{\div}{\operatorname{div}}

\renewcommand\footnotemark{}
\usepackage[pagestyles]{titlesec}
\widenhead*{0.5in}{0.5in}
\renewpagestyle{plain}[\bfseries]{\footrule\setfoot{}{\thepage}{}}
\newpagestyle{mystyle}[\bfseries]{
\headrule
\sethead[\thepage][][Discrepancy and biLipschitz equivalence]{Henna Koivusalo and Jean Lagac\'e}{}{\thepage}
}
\newpagestyle{mystyleappendix}[\bfseries]{
\headrule
\sethead[\thepage][][Number variance estimates]{Michael Bj\"orklund and Tobias Hartnick}{}{\thepage}
}

\title{Appendix A: $\RL^2$ bounds for the discrepancy of Liouville cut and project sets}
\author{Michael Bj\"orklund \thanks{\textbf{MB:} Department of Mathematics, Chalmers,
        Gothenburg, Sweden 
    \href{mailto:micbjo@chalmers.se}{\nolinkurl{micbjo@chalmers.se}}}
\and
Tobias Hartnick \thanks{\textbf{TH:} Institut f\"ur Algebra und Geometrie, KIT, Karlsruhe,
    Germany 
\href{mailto:tobias.hartnick@kit.edu}{\nolinkurl{tobias.hartnick@kit.edu}}}}

\makeatletter
\let\appendxtitle\@title
\let\appendxauthor\@author
\makeatother

\makeatletter
\renewcommand\@maketitle{%
  \newpage
  \null
  \vskip 2em%
  \begin{center}%
  \let \footnote \thanks
    {\LARGE \@title \par}%
    \vskip 1.5em%
    {\large
      \lineskip .5em%
      \begin{tabular}[t]{c}%
        \@author
      \end{tabular}\par}%
      \vskip 1.5em%
    {\large With an appendix by \par}%
    \vskip 1em%
    {\large
      \lineskip .5em%
      \begin{tabular}[t]{c}%
        \appendxauthor%
      \end{tabular}\par}%
      \vskip 1.5em%
  \end{center}%
  \par
  }
\makeatother

\author{Michael Bj\"orklund
\and
Tobias Hartnick}

\makeatletter
\let\appendxauthorbis\@author
\makeatother

\makeatletter
\def\makeappendixtitle{%
  \newpage
  \null
  \vskip 2em%
  \begin{center}%
  \let \footnote \thanks
    {\LARGE \appendxtitle \par}%
    \vskip 1.5em%
    {\large
      \lineskip .5em%
      \begin{tabular}[t]{c}%
        \appendxauthorbis
      \end{tabular}\par}%
      \vskip 1.5em%
  \end{center}%
  \par%
  }%
\makeatother

\begin{document}

\title{
    \begin{centering}
    Sharp density discrepancy for cut and project sets \\ An approach via lattice point counting 
\end{centering}
}

\author{Henna Koivusalo \thanks{\textbf{HK:} School of Mathematics, University of Bristol, Bristol United Kingdom; 
        \href{mailto:henna.koivusalo@bristol.ac.uk}{\nolinkurl{henna.koivusalo@bristol.ac.uk}}}
\and
Jean Lagac\'e \thanks{\textbf{JL}: Department of Mathematics, King's College London, London, United Kingdom; 
\href{mailto:jean.lagace@kcl.ac.uk}{\nolinkurl{jean.lagace@kcl.ac.uk}}}}

\date{}

\maketitle

\abstract{
Cut and project sets
    are obtained by taking an irrational slice of a lattice and projecting it to
    a lower dimensional subspace.
    We seek to quantify fluctuations from the asymptotic mean for point counts. We obtain uniform upper
    bounds on the discrepancy depending on the diophantine properties of the
    lattice.
    In an appendix, Michael Björklund and Tobias Hartnick obtain lower bounds on
    the $\RL^2$ norm of the discrepancy also depending on the diophantine class; these lower bounds match our uniform upper
    bounds and both are therefore sharp. 
    This also allows us to find sharp bounds
    for the number of lattice points in thin slabs.
    Using  a sufficient criteria of Burago--Kleiner and
    Aliste-Prieto--Coronel--Gambaudo we find an explicit full-measure class of
    cut and project sets that are biLipschitz equivalent to
    lattices; our lower bounds 
    indicate that this is the largest
    class of cut and project sets for which those criteria can apply.
 }

\thispagestyle{plain}
\pagestyle{mystyle}

\section{Introduction}
\subsection{Cut and project sets}
A \emph{cut and project set} $\Lambda$ is a discrete subset of euclidean space
 which, despite being aperiodic, still exhibits repetitive structure and
long-term order. Fix an orthogonal decomposition 
\begin{equation}
    \label{eq:decomposition}
\R^d \cong \E := \E_\dwn \oplus
\E_\lef \cong \R^{d_\dwn} \oplus \R^{d_\lef},
\end{equation}
with associated projections
$\pi_\dwn$ and $\pi_\lef$. Fix a lattice $\Gamma \subset
\E$, $\bs \in \E$ and a bounded
\emph{window} $\Omega_\lef \subset \E_\lef$. The cut and project set associated with this
data is 
\begin{equation}
    \Lambda := \Lambda(\E,\E_\dwn,\E_\lef;\Gamma + \bs;\Omega_\lef) := \pi_\dwn((\Gamma + \bs) \cap
    \pi_\lef^{-1}(\Omega_\lef));
\end{equation}
it is said to be \emph{regular} if $\Omega_\lef$ is an open set whose boundary
has zero measure, and $s$-\emph{regular} if, in addition $\del\Omega_\lef$ has a
finite $(d_\lef-s)$-Minkowski content for some $0<s\le1$ (for definitions, see Section \ref{sec:conditions}). We also make the standard assumption that $\pi_\lef(\Gamma)$ is dense in $\E_\lef$ and
$\pi_\dwn\big|_{\Gamma}$ is injective; we give a characterisation of these via \emph{complete irrationality} of $\Gamma$ in Section
\ref{sec:latticesdef}.
Since the decomposition \eqref{eq:decomposition} is fixed we generally drop them from the
notation and write $\Lambda(\Gamma + \bs, \Omega_\lef)$. 
Cut and project sets have received an increasing amount of interest of late as
mathematical models for \emph{quasicrystals}. 
Also referred
to as \emph{model sets}, they originated in mathematics as 
 generalisations of lattices by Meyer in 1970 \cite{Meyer70}.

Many of the standard questions about lattices extend naturally to
cut and project sets; amongst which is the counting problem consisting of finding a
good asymptotic description for the quantity $\#(\Lambda \cap \Omega_\dwn)$
where the search region $\Omega_\dwn
\subset \E_\dwn$ is large, for example a large ball $\B_\dwn(0,t)\subset
\E_\dwn$ or cube $[0,t]^{d_\lef} \subset \E_\dwn$. It is well-known (see, for example, \cite{HaynesKoivusaloWaltonSadun16}, also see \cite[Theorem 7.2]{BaakeGrimmBook}) that 
\begin{equation}
    \label{eqdef:discrepcandp}
    \#(\Lambda \cap \B_\dwn(0,t)) =
    \frac{\vol(\B_\dwn(0,1))\vol(\Omega_\lef)}{\covol(\Gamma)} t^{d_\dwn} +
    \Delta(\Lambda;\B_\dwn(0,t)),
\end{equation}
where the discrepancy is smaller than the main term, i.e.
$\Delta(\Lambda;\B_\dwn(0,t)) =
\smallo{t^{d_\dwn}}$ for every $\Lambda$. Here, the volumes are computed in the
appropriate subspace and the covolume of a lattice $\Gamma$ is defined as
$\covol(\Gamma) := \vol(\E/\Gamma)$. 

The purpose of the current work is to find estimates for the discrepancy $\Delta$
depending on the search region $\Omega_\dwn$, the window $\Omega_\lef$ and the diophantine properties of
$\Gamma$. We establish a direct,
explicit link between a quantifiable diophantine property of the lattice and a
discrepancy bound for the corresponding cut and project set by reducing the
problem to lattice point counting, see Theorem
\ref{thm:main}. We apply the discrepancy estimates to the problem of identifying
new classes of cut and project sets that can be mapped by a biLipschitz map to
a
lattice, see Theorem \ref{thm:mainBL}. In particular, we note that for windows
of finite perimeter the sharp classes we obtain depend only on the diophantine
class of the lattice $\Gamma$ and nothing else.

\subsection{Uniform upper bounds for discrepancy}
We will first state our main theorem on discrepancies of cut and project sets.
We begin by defining what it means for a lattice to be \emph{$\psi$-repellent}
(for the origin of this term, see \cite{BjorklundHartnick}). 
\begin{defi}
    \label{def:repel}
Given an increasing
$\psi: (0, \infty)\to (0,\infty)$ so that $\psi(t)\to \infty$ as $t\to \infty$,
we call a lattice $\Gamma$ $\psi$-repellent (with respect to the decomposition
$\E = \E_\dwn \oplus \E_\lef$), if for every $\gamma \in \Gamma\setminus \{0\}$, 
\[
    |\pi_\lef(\gamma)| > \psi\left(|\pi_\dwn(\gamma)|\right)^{-1}. 
\]
If there is some $\mu > 0$ so that $\psi(r) \asymp r^\mu$ at $\infty$, we say
that $\psi$ \emph{grows at speed $\mu$ at infinity}. If on the other hand
$\psi(r) \lesssim_\eps r^\eps$ for every $\eps > 0$, for instance $\psi(r) =
\log(1 + r)$, we say that $\psi$ \emph{grows slowly at infinity}.
\end{defi}

\begin{theorem}
    \label{thm:main}
    Suppose that $\Gamma \subset \E$ is completely irrational with respect to
    $\E_\lef$ and let $0 <
    \fd \le 1$. Let the window $\Omega_\lef
    \subset \E_\lef$ have boundary with finite
    $(d_\lef-\fd)$-Minkowski content, and assume that the dual lattice
    $\Gamma^{\dagger}$ is $\psi$-repellent for some $\psi: (0, \infty)\to
    (0,\infty)$ that grows slowly. Let $\Omega_\dwn\subset \E_\dwn$ have finite perimeter. 
    
    Then, for any $\delta>0$ there are
$C_{\Lambda,\delta}, t_0>0$ so that for every $\bs \in \E$ and every $t>t_0$, the cut and project set
    $\Lambda(\Gamma+\bs; \Omega_\lef)$
    satisfies, 
    \begin{equation}
        \label{eq:withcubes}
        \abs{ \#(\Lambda\cap t\Omega_\dwn) - \frac{\vol(\Omega_\dwn) \vol(\Omega_\lef)}{\covol(\Gamma)}
        t^{d_\dwn}} \le C_{\Lambda,\delta} t^{d_\dwn} \psi(t)^{-s(1-\delta )}. 
    \end{equation}

\end{theorem}

In Section \ref{sec:lattices} we show that proving Theorem \ref{thm:main}
reduces to a lattice point count in
anisotropically expanding domains --- this is the content of Lemma
\ref{lem:sumrepresentation}. Our main theorems in the lattice point
counting context are Theorems \ref{thm:maingeneral} and \ref{thm:pwswindows} of which Theorem
\ref{thm:main} is then a special case. 
Note that by the Khintcine-Groshev Theorem (also see
\cite{BjorklundHartnick}) the property of being $\psi$-repellent for
$\psi(t)\asymp t^{\mu}$, and hence our results, hold in a conull set of
lattices.

In Theorem \ref{thm:maingeneral} we assume that the search
region is a strictly convex domain satisfying some boundary curvature
assumption described later, whereas Theorem \ref{thm:pwswindows} merely assumes that the
search region has finite perimeter, so that balls
and cubes are special cases, respectively. A
simple
    self-contained proof of the asymptotic \eqref{eqdef:discrepcandp} is included along the way to Theorem \ref{thm:main}. 
We also formulate in Section \ref{sec:acceptance} a corollary of Theorem \ref{thm:main} for discrepancies of frequencies of $r$-patterns in cut and project sets, see Theorem \ref{thm:frequency}.

The main interest of the discrepancy bounds in Theorem \ref{thm:main} is that
generically in a complementary set of lattices we obtain a matching lower bound
along a subsequence, see Subsection \ref{sec:lowerbounds} and the appendix of
this text. This means that the class of lattices for which we obtain
these upper bounds is sharp. Of course other upper bounds on the discrepancy have been proved in the past by
Haynes, Julien, Koivusalo, and Walton \cite{HaynesJulienKoivusaloWalton19} and
R\"uhr, Smilansky and Weiss \cite{RSW}\footnote{To be more precise, these results are concerned with discrepancy to pattern frequencies. However, as we will explain in Section \ref{sec:acceptance}, there is no reason to make a distinction between the two problems.}. Haynes \emph{et al.} had a slightly different
parametrisation of the cut and project scheme as compared to the one we have
here, and one crucial difference is that they take both the window $\Omega_\lef$
and the search region $\Omega_\dwn$ to be axes-aligned cubes. This, together
with a diophantine condition on the lattice $\Gamma$ allows them to deduce a
very small discrepancy, in a similar way to the anomalous small remainders proven
for lattices in \cite[Theorem 3.5]{KordyukovYakovlevSpectralGeometry}. Namely,
they show that for any $\eps>0$, and almost
every choice of $\Gamma$, 
\begin{equation}
        \abs{\#(\Lambda\cap t\Omega_\dwn) -
        \frac{\vol(\Omega_{\lef})\vol(\Omega_\dwn)}{\covol(\Gamma)}t^{d_\dwn}} \le
        C_{\eps, \Gamma} (\log t)^{d_\lef -
    \eps}.
    \end{equation}
That these anomalously low bounds hold can also be deduced from an earlier work of Haynes, Kelly and Weiss \cite{HKW14}, but their parametrisation is yet different. The aim of both of these results was to obtain a conull set of lattices in which the discrepancy bounds hold. In particular, they did not work out the discrepancy bounds for classes broader than what was necessary for the statement on full measure, which amounts to a type of weak badly approximable condition on $\Gamma$. The weak bad approximation properties appearing in these works differ from the $\psi$-repellent condition of Theorem \ref{thm:main}. 
    
    The result of R\"uhr \emph{et al.} is also for $\Gamma$ belonging to a full
    measure subset of the set of lattices, but they
    do allow a wider choice of $\Omega_\lef$ and $\Omega_\dwn$. In particular,
    they assume that $\partial \Omega_\lef$ has box dimension bounded from above
    by $d_\lef-\fd$ and prove that for almost every lattice
\begin{equation}
        \abs{\#(\Lambda\cap \Omega_\dwn) -
        \frac{\vol(\Omega_{\lef})\vol(\Omega_\dwn)}{\covol(\Gamma)}} = \bigo{
    \vol(\Omega_\dwn)^{1-\frac{\fd}{d_\dwn-2\fd}}}. 
    \end{equation}
    For this result also, an underlying generic diophantine condition is key,
    but the techniques obscure it and it seems difficult to track it down
    explicitly.

All of these results give a growth rate for the discrepancy which is smaller
than that of Theorem \ref{thm:main}, but there is no clarity on the behaviour of
the discrepancy on the measure zero (but comeagre) set of lattices on which
their result does not hold. Theorem \ref{thm:main} differs from these
results in two ways. Foremost, it uses lattice point counting, a technique which is new to the study of
discrepancies of cut and project sets. Further, in Theorem \ref{thm:main}, the
connection between the established discrepancy bound and the function $\psi$ is
explicit and, in comparison to \cite{HaynesJulienKoivusaloWalton19} and
\cite{HKW14}, we give a new class of cut and project sets with uniform upper
bounds for discrepancy, as the diophantine condition of Theorem \ref{thm:main}
is explicit and on the dual lattice and not the lattice $\Gamma$ itself.

\subsection{Averaging, and lower bounds for discrepancy}
\label{sec:lowerbounds}
It is
natural to ask what is the best possible bound for the discrepancy that one can
expect. 
The reader may also have
noticed that the translation parameter $\bs$ seems spurrious, as none of the
previous results depended on it. This is the case because the methods we use are
translation invariant. At the same time, translating $\Gamma$ by $\bs$ is the same
as translating the window by $\bs_\lef$ and translating the search region by
$\bs_\dwn$; and up to a small set of parameters this should not affect the
actual statistics for the distribution of the points in a cut and project set.
With this heuristic in mind, we average the absolute value of the discrepancy
over a fundamental domain of $\Gamma$ to obtain lower bounds. 
The next theorem
should be interpreted as saying that this is the best upper bound on the
discrepancy that one can obtain on a full measure set of $\bs$. We note that these lower
bounds hold with even weaker assumptions on the lattice and the window. Indeed,
no irrationality at all is assumed on $\Gamma$.
\begin{theorem}
    Suppose $\Gamma \subset \E$ is any lattice, $\Omega_\lef$ is any bounded
    window with nonzero volume. There is $C, t_0 > 0$ so that the cut and project 
    sets $\Lambda(\Gamma+\bs,\Omega_\lef)$ satisfy for all $t > t_0$
    \begin{equation}
        \label{eq:lowerbound}
        \int_{\E/\Gamma} \abs{\Delta(\Lambda(\Gamma + \bs,\Omega_\lef);\B_\dwn(0,t))} \de \bs
        \ge C f(t) t^{\frac{d-1}{2}}.
    \end{equation}
    where for some $A > 0$ the function $f(t)$ is defined as
    \begin{equation}
        f(t) = \begin{cases}
            1 & \text{if } d \not \equiv 1 \mod 4 \\
            \exp (-A \log \log(t)^4) & \text{if } d \equiv 1 \mod 4.
        \end{cases}
    \end{equation}
    Furthermore,
    \begin{equation}
        \int_{\E/\Gamma} \Delta(\Lambda(\Gamma + \bs,\Omega_\lef);\B_\dwn(0,t))
        \de \bs = 0.
    \end{equation}
    \label{thm:lowerbound}
\end{theorem}
Just as Theorem \ref{thm:main} was a special case of Theorem
\ref{thm:maingeneral} through a characterisation as a lattice point counting
problem as proved in Section \ref{sec:lattices}, Theorem \ref{thm:lowerbound} is
a consequence of a corresponding lattice result, Theorem \ref{thm:loweravg}. 
What Theorem \ref{thm:lowerbound} indicates is that for every cut and project set and every $t$
large enough, there are positive measure subsets of $\E/\Gamma$ so that the
remainder terms are as big as indicated in \eqref{eq:lowerbound}, both in the
positive and the negative direction. This is, in a way, as best as one can hope
to do on full measure sets. In the Appendix, written by Michael Bj\"orklund and Tobias Hartnick,
it is instead the $\RL^2$ norm in the parameter $\bs$ which is considered, also
known as the  number variance. More precisely, define a lattice to be
$\psi$-Liouvillean if it satisfies the complementary property to being
$\psi$-repellant; in other words if there is a sequence $\gamma^{(n)}$ so that
\begin{equation}
    |\pi_\lef(\gamma^{(n)})| \le \psi(|\pi_\dwn(\gamma^{(n)})|^{-1}).
\end{equation}
Bj\"orklund and Hartnick show that for almost every $r > 0$, every cut and
project set $\Lambda$ originating from the window $\B_\lef(0,r)$ and
lattice whose dual is $\psi$-Liouvillean satisfies the asymptotics
\[
\limsup_{t\to\infty}\frac{\int_{\E/\Gamma} |\Delta(\Lambda(\Gamma+\bs,
\B_\lef(0,r));t\Omega_\dwn)|^2\de \bs}{t^{2d}\psi(t)^{-d_\lef - 1 - \delta}}=\infty. 
\]
for any $\delta > 0$. This means that, at least along subsequences, the
discrepancy can be very large. In fact, they also show that there exist
explicit $\psi$-Liouvillean lattices and $\Omega_\lef$ so that cut and project
sets $\Lambda$ with this data satisfies
\begin{equation}
    \label{eq:lownv}
\limsup_{t\to\infty}\frac{\int_{\E/\Gamma} |\Delta(\Lambda(\Gamma+\bs,
\Omega_\lef); t\Omega_\dwn)|^2\de \bs}{t^{2d_\dwn}\psi(t)^{-2 - \delta}}=\infty
\end{equation}
for any $\delta > 0$; see Theorems \ref{ThmLiouville} and \ref{ThmLiouvillebis}. For those lattices satisfying
\eqref{eq:lownv} H\"older's inequality implies that 
    \begin{equation}
        \label{eq:withcubessharp}
        \limsup_{t \to \infty} \frac{|\Delta(\Lambda(\Gamma + \bs,\Omega_\lef);t
        \Omega_\dwn)|}{t^{d_\dwn} \psi(t)^{-1 + \delta}}   = \infty. 
    \end{equation}
so that Theorem \ref{thm:main} is sharp when the boundary of the window has
Minkowski dimension $d - 1$.

\subsection{Alternative interpretation: lattice points in thin slabs}

There is another geometric interpretation of our results in terms of lattice
point counts which we describe now. Counting the number of points in the cut and
project set $\Lambda(\Gamma + \bs,\Omega_\lef)$ within dilates $t \Omega_\dwn$
of a search region is the same as counting the number of points in the lattice
$\Gamma$ in a thin slab $(t\Omega_\dwn \times \Omega_\lef) - \bs$. Our main
result, Theorem \ref{thm:main} combined with Theorem \ref{ThmLiouvillebis} can be rewritten in this language.
\begin{theorem}
    \label{thm:latticepointcountingversion}
    Let $\E = \E_\dwn \oplus \E_\lef$, and $\Gamma$ be a lattice completely
    irrational with respect to this decomposition. Let $\Omega_\lef \subset
    \E_\lef$ and $\Omega_\dwn \subset \E_\dwn$ have finite perimeter, and assume
    furthermore that the dual lattice $\Gamma^\dagger$ is $\psi$-repellent for
    some slowly growing $\psi : (0,\infty) \to (0,\infty)$. Then, for any
    $\delta > 0$ there is $C_{\Gamma,\delta}$ so that
    \begin{equation}
        \left|\# \Gamma \cap (t \Omega_\dwn \times \Omega_\lef) -
        \frac{\vol(\Omega_\dwn)\vol(\Omega_\lef)}{\covol(\Gamma)}t^{d_\dwn}
        \right| \le C_{\Gamma,\delta} t^{d_\dwn} \psi(t)^{-1 + \delta}.
    \end{equation}
    Furthermore, there are $\psi$-Liouvillean lattices and sets $\Omega_\lef
    \subset \E_\lef$ and a sequence $t_n \to \infty$ so that for every $\delta >
    0$, 
    \begin{equation}
        \left|\# \Gamma \cap (t_n \Omega_\dwn \times \Omega_\lef) -
        \frac{\vol(\Omega_\dwn)\vol(\Omega_\lef)}{\covol(\Gamma)}t_n^{d_\dwn}
        \right| \ge t_n^{d_\dwn} \psi(t_n)^{-1 - \delta}.
    \end{equation}
\end{theorem}
We note that this refines and corrects the lattice point counting results in
anisotropically expanding domains found in any of
\cite{KordyukovYakovlevZd,KordyukovYakovlevSpectralGeometry,lagace}, where the
contribution from Diophantine properties of the lattice was overlooked.

A classical problem in number theory is to prove existence of approximate
integer solutions to polynomial equations, for instance Davenport and Heilbronn
proved in \cite{DavenportHeilbronn} that there are infinitely many solutions to
\begin{equation}
    \label{eq:sumofsquares}
    \sum_{j=1}^5 \lambda_j n_j^2 \le \eps
\end{equation}
provided the $\lambda_1, \dots, \lambda_5$ are not zero, do not all share the same sign, and are
pairwise rationally independent. Freeman obtained asymptotics for the number
of solutions where every coordinate $n_i$ is below some value~$T$~\cite{Freeman}. This
corresponds to counting the number of points of $\Z^d$ in $\ker L \oplus
[-\eps,\eps]$ where $L$ is the linear map given by the matrix
$\begin{pmatrix} \lambda_1 & \lambda_2 & \lambda_3 & \lambda_4 & \lambda_5
\end{pmatrix}$ under the additional constraint that all coordinates are squares. After a
linear change of variable, this can be interpreted in our setting, up to the additional
constraint on the coordinates being squares.

Recently, in \cite{Walker}, Walker has obtained results of this type for general matrices $L$ in the case where the
constraint is that all the coordinates belong to some set $A$ of positive
density, as opposed to the situation in \eqref{eq:sumofsquares} where the set of
squares has density zero.
In this situation, he obtains an asymptotic relation where the first term is given by the
count of \emph{all} lattice points in $\ker L \times [-\eps,\eps]$, plus an
error term which is bounded in terms of the set $A$ and Diophantine properties
of the (dual) lattice \cite[Lemma 3.4]{Walker}; see also \cite[Lemma 8.4]{Walker2} where
the coefficients are assumed algebraic. As such, Theorem \ref{thm:latticepointcountingversion} makes
more precise, and sharp, the contribution of the principal term in Walker's results.

\subsection{BiLipschitz equivalence to lattices}

We now turn our attention to the question of mapping cut and project sets onto
lattices with a biLipschitz map $\phi: \Lambda \to \Z^d$. If such a map exists,
we say that $\Lambda$ is \emph{biLipschitz (BL) equivalent} to a lattice. The
question of BL equivalence to lattices of relatively dense and uniformly
discrete sets (\emph{separated nets}) is an old problem  originating in
geometry: if two metric spaces have separated nets that
are BL equivalent to each other, then they are quasi-isometric. Gromov \cite{G93} asked
whether there are choices of separated nets in the same metric space that are
not BL equivalent to each other. The constructions of Burago and Kleiner
\cite{BK98} and McMullen \cite{M98} were the first to show that the answer is
`yes' in euclidean space: there are 
separated nets that are not BL equivalent to a lattice.
Later it was found that in fact the collection of BL classes of separated nets
in $\R^d$ has the size of the continuum \cite{M11}. Burago and Kleiner
\cite{BK02} (in $\R^2$), and later, Aliste-Prieto, Coronel and Gambaudo
\cite{ACG13} (in $\R^d$) also showed that BL
equivalence can be reduced to uniform point counting in large cubes.

In the case of cut and project sets, Burago and Kleiner \cite{BK02} were able to prove that for almost every choice of $\Gamma$, a regular cut and project set in
dimension $d=3$ is BL equivalent to a lattice. This prompted them to
ask whether every regular cut and project set has this property. Their result was
extended to $s$-regular cut and project sets in any dimension $d$ by Haynes,
Kelly and Weiss \cite{HKW14}. Proofs of this kind rely on a result given in \cite[Theorem 3.1]{ACG13}, see also
\cite{BK02}, which establishes a link between discrepancy estimates of cut and project sets,
and their BL equivalence class. In Theorem \ref{thm:main} we established an explicit connection between a diophantine property of $\Gamma^\dagger$ and the discrepancy of the corresponding cut and project set. We can now apply it to find the following new class of cut and project sets which are BL equivalent to lattices. 

%
%

\begin{theorem}\label{thm:mainBL}
Suppose that $\Gamma \subset \E$ is completely irrational with respect to
    $\E_\lef$, $\bs\in \E$, and assume that $\Omega_\lef
    \subset \E_\lef$ is $s$-regular for some $0<s<1$. Assume, furthermore, that
    there exists some $\eta > 0$ such that $\Gamma^\dagger$ is $\log (1+t)^{(1 +
    \eta)/s}$-repellant. Then, the cut and project set $\Lambda(\Gamma + \bs,
    \Omega_\lef)$ is BL equivalent to a lattice.
\end{theorem}
As pointed out above, the property of being $\psi$-repellent for $\psi(t)\asymp
t^{\mu}$ holds in a conull set of lattices, and in particular the above result
establishes that almost every lattice $\Gamma$ corresponds to a cut and project
set which is BL equivalent to a lattice. We note that Theorems
\ref{ThmLiouville} and \ref{ThmLiouvillebis} of the Appendix show that this is
more or less the largest class of cut and project to which the sufficient conditions  of
Burago--Kleiner and Aliste-Prieto--Coronel--Gambaudo can be applied, up to a
very small (both null and meagre) set. In other words, if we want to determine
the
biLipschitz class of more cut and project sets a weaker sufficient condition
needs to be found, or an example which is not biLipschitz equivalent
to a lattice.

\subsubsection*{Plan of the paper}

In Section \ref{sec:background} and \ref{sec:candp}, we cover the background on
lattices, the Fourier transform and cut and project sets. There, we also fix
some notation and discuss the irrationality assumptions that are necessary for
our results to hold. Section \ref{sec:lattices} connects the distribution of
points in cut and project sets to lattice point counting, and states the main
general theorems. Section \ref{sec:smoothing} is a preparation section; we are
aiming to use Poisson summation for the lattice point counts. As such, it is
necessary to find smooth approximations to the functions being summed in order
to ensure sufficient decay of the Fourier transforms.
Section \ref{sec:proofsupper} is concerned with the proof of Theorems \ref{thm:maingeneral}
which gives upper bounds to discrepancy. Section
\ref{sec:quantitative} is again a preparation section, this time providing
uniform diophantine estimates for lattice projections. In Section
\ref{sec:average} we prove the Theorem \ref{thm:loweravg} on the average lower
bound for the discrepancy. Finally, the appendix of Bj\"orklund and Hartnick
gives the lower bounds on the number variance.

\subsubsection*{Acknowledgements}

The authors are grateful to A. Haynes, Y. Smilansky, J. Walton and B. Weiss for
their comments on an early version of this manuscript. The authors would also
like to thank J. Marklof for useful discussions. This article
profited from discussions during the Workshop on \emph{Aperiodic order and
approximate groups}  at KIT during August 2023. We thank KIT for the funding of
the workshop and are grateful to the participants for their input.
JL is thankful for the hospitality of the University of Bristol at the time of
writing this paper.

\section{Notation, definitions and assumptions}\label{sec:background}

In this section we collect notations, definitions that we use in this paper, as
well as recall some facts about them. The results stated in this section are
standard or folklore, we state them for convenience of reference. As such, the reader who is
short on time may wish to pass directly to Section \ref{sec:lattices}.

\subsection{Euclidean space}

Total space is denoted by $\E$, and $\dim(\E) = d$. The canonical decomposition
with respect to which cut and project sets are defined is always denoted 
\begin{equation}
    \E = \E_\dwn \oplus \E_\lef,
\end{equation}
with $\dim(\E_\dwn) = d_\dwn$ and $\dim(\E_\lef) = d_\lef = d - d_\dwn $. Sometimes, we require
other subspaces or decompositions of $\E$, in which case we denote them $\F
\subset \E$, or $\E = \F_\dwn \oplus \F_\lef$. The projection on $\F$ is denoted
by $\pi_\F$. 

Given a decomposition, whether the canonical one or an auxiliary one, we always
indicate with subscripts $\dwn$ or $\lef$ the association of a variable, a
subset, an operator, etc. with one of the subspace, for instance we would write
$\pi_\dwn$ for $\pi_{\E_\dwn}$. We also indicate with
$\bullet \in \set{\dwn,\lef}$ an association with either of them.

When describing the volume of a set, we will not always specify it is the volume
with respect to which subspace and it will be clear from context. For instance,
for some $\Omega_\dwn \subset \E_\dwn$, $\vol(\Omega_\dwn)$ is the
$d$-dimensional volume of $\Omega_\dwn$ induced by the Lebesgue measure on
$\E_\dwn$.

We give special names to some subsets of the fixed decomposition $\E = \E_\dwn
\oplus \E_\lef$

\begin{defi}
    A \emph{window} is a bounded subset $\Omega_\lef \subset \E_\lef$. If its interior
    is equal to the interior of its closure we say that it is a \emph{regular
    window}.
\end{defi}

\begin{defi}
    A \emph{search region} is a finite perimeter subset $\Omega_\dwn \subset \E_\dwn$. 
\end{defi}

We denote the indicator function of a set $\Omega \subset \E$ by $\chi_\Omega : \E \to \R$
\begin{equation}
    \chi_\Omega(\bx) = \begin{cases}
        1 & \text{if } \bx \in \Omega \\
        0 & \text{otherwise.}
    \end{cases}
\end{equation}


\subsection{Asymptotic notation}

We make frequent use of the Landau and Vinogradov asymptotic notation. To wit,
given two real-valued functions $f,g$, we say that
\begin{itemize}
    \item indiscriminately, $f = \bigo{g}$, $f \lesssim g$, or $g \gtrsim f$ to indicate that
        there exists $C > 0$ so that $\abs f \le C\abs g$;
    \item we use $f \asymp g$ to indicate that $f \lesssim g$ and $g \lesssim
        f$. 
\end{itemize}
In either cases, an index on the asymptotic notation indicates dependence of the
constants on this index. For instance, writing $f \lesssim_{\Gamma,\fd} g$ means
that the constant $C$ in the definition above may depend on $\Gamma$ and $\fd$.
\subsection{Lattices}
\label{sec:latticesdef}
\begin{defi}
    Given $\F \subset \E$ with potentially $\F = \E$, a \emph{lattice in $\F$} is
    a discrete cocompact subgroup $\Gamma$ of the group of translations
in $\F$. Through the identification of $\F$ with its group of translations, we
view a lattice as a subset $\Gamma \subset \F$, and $\F/\Gamma$ is a torus which
we can identify with any parallelotope (called a fundamental cell) generated by a basis of $\Gamma$. Given
a lattice this identification is fixed once and for all, the fundamental cell is
also denoted $\F/\Gamma$, and every $\bx \in 
\F$ can therefore be written uniquely as $\bx = \gamma + \bk$, for some $\gamma
\in \Gamma$ and $\bk \in \F/\Gamma$. 
The covolume of a lattice
$\Gamma \subset \F$ is defined as
\begin{equation}
    \covol(\Gamma) := \vol(\F/\Gamma).
\end{equation}
\end{defi}

The space $\CX_\E$ of all lattices in $\E$ is identified with
\begin{equation}
    \CX_\E := \GL(d,\Z)\backslash \GL(d,\R) / S_d,
\end{equation}
where $S_d$ is the symmetric group acting by permuting the columns of a $d\times
d$ matrix. Note that
this may differ from the standard description found in the litterature where
$S_d$ would be replaced by $\mathrm{SO}(d,\R)$; this is because for our purposes
we do \emph{not} want to identify lattices which are equivalent up to orthogonal
transformations. From this identification we see that $\CX_\E$ is an orbifold of
dimension $d^2$, and it
inherits a (Haar) measure and a topology from $\GL(d,\R)$, on which we put the
topology induced by the operator norm.

\begin{defi}
    Let $\Gamma$ be a lattice. Its dual lattice $\Gamma^\dagger$ is defined as
    \begin{equation}
        \Gamma^\dagger := \set{\gamma^\dagger \in \E : \gamma^\dagger \cdot
        \gamma \in \Z \text{ for all } \gamma \in \Gamma}.
    \end{equation}
    The dual lattice is itself a lattice, and $(\Gamma^\dagger)^\dagger =
    \Gamma$.
\end{defi}

\begin{defi}
    \label{def:gammasubspace}
    Let $\F \subset \E$ be a subspace. We denote $\Gamma(\F) := \Gamma \cap \F$
    and we say that $\F$ is a \emph{$\Gamma$-subspace} if $\Gamma(\F)$ spans
    $\F$.
\end{defi}
We note the important distinction between $\Gamma^\dagger(\F) := \Gamma^\dagger
\cap \F$ and
\begin{equation}
    \Gamma(\F)^\dagger := \set{\gamma^\dagger \in \spn(\Gamma(\F)) : \gamma^\dagger \cdot 
    \gamma \in \Z}.
\end{equation}

There is a correspondence between $\Gamma$-subspaces of dimension $m$ and
$\Gamma^\dagger$-subspaces of codimension $m$. Indeed, if $\F \subset \E$ is a
$\Gamma$-subspace, $\F^\perp$ is a $\Gamma^\dagger$-subspace, see
\cite[Lemma 3.2]{KordyukovYakovlevSpectralGeometry}. This correspondence is
showcased in the following lemma.
\begin{lemma}
    \label{lem:gammadaggersubspace}
    Let $\Gamma \subset \E$ be a lattice. For every $x \in \E$, $\spn(x)$ is a
    $\Gamma^\dagger$-subspace if and only if there exists $t \in \R$ and $\delta >
    0$ so that
    \begin{equation}
        \label{eq:hb}
        \set{y \in \E : x \cdot y \in (t-\delta,t+\delta)} \cap \Gamma =
        \varnothing.
    \end{equation}
\end{lemma}

\begin{proof}
    We first suppose $\spn(x)$ is a $\Gamma^\dagger$-subspace, so that there
    exists $a \in \R\setminus\{0\}$ such that $ax \in \Gamma^\dagger$. By definition of
    $\Gamma^\dagger$,
    \begin{equation}
        \set{y \in \E : ax \cdot y \in (1/4,3/4)} \cap \Gamma = \varnothing.
    \end{equation}
    Therefore we can take $t = (2a)^{-1}$ and $\delta = (4a)^{-1}$. 

    Suppose on the other hand that \eqref{eq:hb} holds for some $x\in \E$. This implies we cannot make the quantity $\abs{x \cdot \gamma}$ arbitrarily
    small. In particular, this implies that $\abs{x \cdot \gamma}$ for $\gamma\in \Gamma$ is quantised:
    there exists $a > 0$ such that for every $\gamma \in \Gamma$ there exists $n
    \in \Z$ such that ${x \cdot \gamma} = a n$. But then $a^{-1}x \in \Gamma^\dagger$, in other words $\spn(x)$ is a
    $\Gamma^\dagger$-subspace.
\end{proof}

\begin{defi}
    Let $\Gamma \subset \E$ be a lattice and $\F \subset \E$ a subspace. We say
    that $\Gamma$ is \emph{irrational with respect to $\F$} if $\Gamma(\F) = \set
    0$. If $\Gamma^\dagger$ is also irrational with respect to $\F$ we say that
    $\Gamma$ is \emph{completely irrational with respect to $\F$}. 
\end{defi}
%

Given a subspace $\F$ of codimension at least $1$, the set of lattices which are
not completely irrational with respect to $\F$ is dense in the set of lattices, but
is still small, as seen in the following proposition.

\begin{prop}
    \label{lem:sizeexceptions}
    Let $\F \subset \E$ be a subspace of dimension $m < d$. Then, the set of
    lattices completely irrational with respect to $\F$ is a comeagre set whose
    complement has Hausdorff dimension $dm$. 
\end{prop}

\begin{proof}
    The projection by the quotient map of a comeagre set is comeagre, so we can
    identify $\Gamma$ with any of its (unordered) bases in
$\GL(d,\R)/S_d$. In turn, this is viewed as an open subset
    of $\R^{d \times d}/S_d$. The first basis element $\theta_1$ of $\Gamma$ can be chosen
    arbitrarily outside $\F$, that is, outside a set of dimension $m$. Recursively, to ensure that $\F \cap \Gamma = \set 0$,
    the $j$th basis vector needs to be chosen outside
    $\spn_\Z(\theta_1,\dotsc,\theta_{j-1})+\F$. This is a $d$-fold product of
    a countable union of dimension $m < d$ affine conditions, hence the set of
    irrational lattices with respect to $\F$ is a comeagre set whose complement
    has Hausdorff dimension $dm$. 

    To see that the same holds for completely irrational lattices, we observe
    that the involution $\Gamma \mapsto
\Gamma^\dagger$ corresponds in $\GL(d,\R)/S_d$ to the map $A \mapsto
(A^{-1})^\top$. This is a locally Lipschitz homeomorphism, hence it preserves
comeagreness and Hausdorff dimension. By the same construction as above, the set
of lattices whose dual is irrational with respect to $\F$ is also a comeagre set
whose complement has Hausdorff dimension $dm$, taking their union
yields our claim.
\end{proof}

\begin{remark}
    In our setting we will use (complete) irrationality with respect to
    $\E_\lef$. Irrationality of $\Gamma$ with respect to
    $\E_\lef$ implies that the projection $\pi_\dwn$ restricted to $\Gamma$ is injective; this allows
    us to translate the problems at hand into lattice point counting problems
    (see Lemma \ref{lem:sumrepresentation}).
    Irrationality of $\Gamma^\dagger$ with respect to $\E_\lef$ implies that
    $\Gamma_\lef$ is dense in $\E_\lef$ (see Lemma \ref{prop:irrational}). This latter fact is a
    standard condition in the definition of cut and project sets and it will
    appear technically in our use of the Poisson summation formula.
\end{remark}

%
%
%

\subsection{Fourier Transforms of indicator functions}

Define $\be : \R \to \C$ as 
\begin{equation}
    \be(x) := e^{-2 \pi i x}.
\end{equation}
The Fourier transform of a function $f : \E \to \C$ is defined as
\begin{equation}
    [\CF f](\bxi) := \int_\E \be(\bx \cdot \bxi) f(\bx) \de \bx
\end{equation}
Given a decomposition $\E = \F_\dwn \oplus \F_\lef$, we also define partial
Fourier transforms relative to the decomposition via the formulae
\begin{equation}
    [\CF_{\dwn} f](\bx_\lef,\bxi_\dwn) := \int_{\E_\dwn} \be(\bx_\dwn \cdot \bxi_\dwn)
    f(\bx_\dwn,\bx_\lef) \de \bx_\dwn
\end{equation}
and
\begin{equation}
    [\CF_{\lef} f](\bxi_\lef,\bx_\dwn) := \int_{\E_\lef} \be(\bxi_\lef \cdot \bx_\lef) 
    f(\bx_\dwn,\bx_\lef) \de \bx_\lef.
\end{equation}
A direct computation tells us that $\CF = \CF_\dwn \CF_\lef = \CF_\lef \CF_\dwn$.
We now provide three estimates on the decay of the Fourier transform of indicator
functions, depending (in decreasing order of regularity) on whether the boundary
of a set is strictly convex, has finite perimeter, or finite upper Minkowski
content. We note that the second estimate could be proved by means of the third,
but since the proof is simpler we keep it for clarity of exposition, in
particular for readers who are only interested in the case where the boundary is
regular.

\begin{lemma}
    \label{lem:Fourierdecay}
    Let $\E = \F_\dwn \oplus \F_\lef$ be a decomposition of Euclidean space with
    $\dim(\F_\dwn) = m$. Let $\Omega \subset \E$ be compact and assume that for every
    $\bx_\lef \in \F_\lef$, $\Omega \cap (\F_\dwn + \bx_\lef)$ either has
    zero $m$-dimensional measure or is a convex set with $\RC^{\frac{m-1}{2}}$
    boundary and principal curvatures uniformly bounded away from $0$. There exists $C > 0$ such that 
    \begin{equation}
        \label{eq:fourierdecay}
        \abs{[\CF \chi_{\Omega}](\bxi)} \le C
        (1 +
        \abs{\bxi_\dwn})^{\frac{-m-1}{2}}.
    \end{equation}
\end{lemma}

\begin{remark}
    We note that the hypotheses of this statement differ slightly from those
    found in \cite[Theorem 4]{KordyukovYakovlevZd}, \cite[Theorem
    2.3]{KordyukovYakovlevSpectralGeometry} (where this lemma is used in their
    proof without explicitly stating it separately) and \cite[Lemma
    6.4]{lagace}. We add here the hypotheses that the boundary of $\Omega$ is sufficiently
    regular and that the principal curvatures are uniformly bounded away from
    $0$. These conditions are necessary to ensure the decay
    \eqref{eq:fourierdecay}: if the principal curvatures vanish at some point
    then the optimal order of decay is worse, see e.g. \cite{randollatticeI} for
    $\RL^p$ balls in $\E$ and \cite[Théorème 1]{CDVpointsentiers} for
    general sets. As such, the hypotheses on regularity and geometry of the
    previous lemma should be added to the statements of the
    aforementioned theorems for their proof to go through in their current
    state.
\end{remark}

\begin{proof}
    We write the Fourier transform as
    \begin{equation}
        \begin{aligned}
        [\CF\chi_\Omega](\bxi) &=  \int_{\F_\lef} \be(\bx_\lef \cdot \bxi_\lef)
        \int_{\F_\dwn} \be(\bx_\dwn \cdot \bxi_\dwn) \chi(\bx_\dwn,\bx_\lef)
        \de \bx_\dwn \bx_\lef \\
        &=  [\CF_\lef[\CF_\dwn \chi_{ \Omega \cap (\F_\dwn +
        \bx_\lef)}]](\bxi).
        \end{aligned}
    \end{equation}
    By the assumptions on $\Omega
    \cap (\F_\dwn + \bx_\lef)$, it follows from 
     \cite[Theorem 2.29]{iosevichliflyand} with $\alpha =
    0$ in their notation that under the hypotheses of the lemma
    \begin{equation}
        \abs{[\CF_\dwn \chi_{\Omega \cap (\F_\dwn + \bx_\lef)}(\bxi_\dwn)} \lesssim 
        (1 +
        |\bxi_\dwn|)^{\frac{-m-1}{2}},
    \end{equation}
    uniformly with respect to $\bx_\lef$. The claim then follows by bounding the
    Fourier transform with respect to the $\F_\lef$ direction by its value at
    $0$. 
\end{proof}

\begin{lemma}
    \label{lem:badFourierdecay}
    Let $\E = \F_\dwn \oplus \F_\lef$ be a decomposition of Euclidean space with
    $\dim(\F_\dwn) = m$. Let $\Omega \subset \E$ be compact and assume that for
    every $\bx_\lef \in \F_\lef$, $\Omega \cap (\F_\dwn + \bx_\lef)$ either has
    zero measure or is a
    $m$-dimensional domain with finite perimeter, and assume that this perimeter
    is uniformly bounded with respect to $\bx_\lef$. There
    exists $C > 0$ such that
    \begin{equation}
        \label{eq:badfourierdecay}
        \abs{[\CF \chi_{\Omega}](\bxi)} \le C
        (1 +
        |\bxi_\dwn|)^{-1}.
    \end{equation}
\end{lemma}

\begin{proof}
    Assume without loss of generality that $\abs{\bxi_\dwn} \ge 1$. Write the
    Fourier transform again as   
    \begin{equation}
        \begin{aligned}
        [\CF\chi_\Omega](\bxi) &=  \int_{\F_\lef} \be(\bx_\lef \cdot \bxi_\lef)
        \int_{\F_\dwn} \be(\bx_\dwn \cdot \bxi_\dwn) \chi(\bx_\dwn,\bx_\lef)
        \de \bx_\dwn \bx_\lef \\
        &=  [\CF_\lef[\CF_\dwn \chi_{ \Omega \cap (\F_\dwn +
        \bx_\lef)}]](\bxi),
        \end{aligned}
    \end{equation}
    where we interpret the indicator function of $\Omega$ as a family
    parametrised by $\bx_\lef$ of indicators in $\F_\dwn$. Applying the
    Gauss--Green theorem
    we have that
    \begin{equation}
        \begin{aligned}
        \abs{\int_{\Omega \cap (F_\dwn + \bx_\lef)} \be(\bx_\dwn \cdot \bxi_\dwn)
        \de \bx_\dwn} &= \abs{\int_{\Omega \cap (F_\dwn + \bx_\lef)}
        \frac{1}{|\bxi_\dwn|^2} \div (\grad_{\bx_\dwn} \be(\bx_\dwn \cdot
        \bxi_\dwn))
        \de \bx_\dwn}\\
    &=  \abs{\int_{\del\Omega \cap (\F_\dwn + \bx_\lef)}
            \frac{\bxi_{\dwn} \cdot \nu}{\abs{\bxi_\dwn}^2} \be(\bx_\dwn \cdot
        \bxi_\dwn) \de \bx_\dwn} \\
        &\lesssim \abs{\bxi_\dwn}^{-1} \vol_{m-1} (\del\Omega \cap
        (\F_\dwn + \bx_\lef)),
    \end{aligned}
    \end{equation}
    where $\nu$ is the unit outward normal to the boundary. Uniform boundedness
    of the perimeter implies our claim.
\end{proof}

%

For our final estimate we start with a definition of (upper) Minkowski content.

\begin{defi}
    Let $\Upsilon \subset \E$ be a bounded set and let $0 \le s \le d$. The
    \emph{$s$-codimensional (upper) Minkowski content of $\Upsilon$ with respect to $\E$} is
    defined as
    \begin{equation}
        \mathfrak{M}_s(\Upsilon;\E) :=  \limsup_{r \searrow 0} r^{-s}
        \vol(\Upsilon_r),
    \end{equation}
    where $\Upsilon_r$ is the $r$-neigbourhood of $\Upsilon$ in $\E$.
\end{defi}

\begin{lemma}
    \label{lem:verybadFourierdecay}
    Let $\E = \F_\dwn \oplus \F_\lef$ be a decomposition of Euclidean space with
    $\dim(\F_\dwn) = m$. Let $0 < s < 1$ and let $\Omega \subset \E$ be compact and assume that for
    every $\bx_\lef \in \F_\lef$, $\Omega \cap (\F_\dwn + \bx_\lef)$ either has
    zero measure or is a
    $m$-dimensional domain whose boundary has finite $s$-codimensional Minkowski
    content with respect to $\F_\dwn$, and assume furthermore that this
    Minkowski content
    is uniformly bounded with respect to $\bx_\lef$. There
    exists $C > 0$ such that
    \begin{equation}
        \label{eq:verybadfourierdecay}
        \abs{[\CF \chi_{\Omega}](\bxi)} \le C
        (1 +
        |\bxi_\dwn|)^{-s}.
    \end{equation}
\end{lemma}

\begin{proof}
    Let $\Omega_\dwn \subset \F_\dwn$ be a set whose boundary has finite
    $s$-codimensional Minkowski content with respect to $\F_\dwn$. For $r > 0$, cover the
    boundary of $\Omega_\dwn$ with finitely many disjoint cubes of sidelength $r
    > 0$, of which we need at most $\lesssim r^{s-m}$. Define $\Omega^{(r)}$ to
    be $\Omega$ with those cubes removed; this is a set with Lipschitz boundary,
    hence to which we can apply the Gauss--Green theorem. Furthermore,
    $\vol_{m-1}(\del \Omega^{(r)}) \lesssim r^{s-1}$. Consequently,
    \begin{equation}
        \begin{aligned}
            \abs{\int_{\Omega_\dwn} \be(\bx_\dwn \cdot \bxi_\dwn) \de \bx_\dwn}
            &\le \abs{\int_{\Omega \setminus \Omega^{(r)}} \be(\bx_\dwn \cdot
            \bxi_\dwn) \de \bx_\dwn} + \abs{\int_{\del \Omega^{(r)}}
            \frac{\bxi_\dwn \cdot \nu}{\abs{\bxi_\dwn}^2} \be(\bx_\dwn \cdot
        \bxi_\dwn) \de \bx_\dwn} \\
        &\lesssim r^s + |\bxi_\dwn|^{-1} r^{s-1}.
        \end{aligned}
    \end{equation}
    We now assume without loss of generality that $|\bxi_\dwn| \ge 1$ is large
    enough and we
    take $r = |\bxi_{\dwn}|^{-1}$ to obtain
    \begin{equation}
        r^s + |\bxi_\dwn|^{-1} r^{s-1} \lesssim |\bxi_\dwn|^{-s}.
    \end{equation}
    Uniformity of the Minkowski content allows us to decompose the integral in
    the $\lef$ and $\dwn$ direction as in the proof of Lemma
    \ref{lem:badFourierdecay}, whence we obtain our desired result.
\end{proof}

\section{Cut and project sets}
\label{sec:candp}

In this section we introduce the definitions and preliminary results for cut and
project sets, and prove Theorem \ref{thm:mainBL} as a consequence of Theorem
\ref{thm:main}. We reserve the letter $\Lambda$ for cut and project sets, and
$\RM$ for subsets
of $\Lambda$ which are themselves a cut and project set.
\begin{defi}
    Let $\E = \E_\dwn \oplus \E_\lef$ be a decomposition of $d$-dimensional
    Euclidean space. Let $\Gamma \subset \E$ be a lattice, $\Omega_\lef \subset
    \E_\lef$ be a window and $\bs \in \E$. The cut and project set associated
    with this data is
    \begin{equation}
        \Lambda(\Gamma;\bs;\Omega_\lef) := \pi_\dwn( ( \Gamma + \bs) \cap
        \pi_\lef^{-1}(\Omega_\lef))
    \end{equation}
\end{defi}
We fix the decomposition $\E = \E_\dwn \oplus \E_\lef$ throughout and this is
why it has been dropped from the notation in this definition of cut and project
sets. 
\begin{remark}
    \label{rem:sinvariant}
    If $\gamma \in \Gamma$, then $\Lambda(\Gamma;\bs;\Omega_\lef) =
    \Lambda(\Gamma;\bs + \gamma;\Omega_\lef)$, in other words the map $\bs \mapsto
    \Lambda(\Gamma;\bs;\Omega_\lef)$ is $\Gamma$-periodic, so that we may
    consider $\bs \in \E/\Gamma$ is an element of the torus rather than of $\E$. In addition, we can
    observe that for any $\bs \in \E$,
    \begin{equation}
        \Lambda(\Gamma;\bs;\Omega_\lef) = \Lambda(\Gamma;0;\Omega_\lef +
        \bs_\lef) + \bs_\dwn,
    \end{equation}
    in other words every cut and project set with a parameter $\bs \in \E$ can
    be realised as the translation of a cut and project set with a parameter
    $\bs = 0$ by translating the window appropriately.
\end{remark}

\subsection{Assumptions on the cut and project set}

\label{sec:conditions}
We do not need to give any conditions on the parameter $\bs$ for our results to hold. We will require some
conditions on the other pieces of data which we list now. 
\begin{itemize}
    \item[\textbf{(O)}] The decomposition $\E_\dwn \oplus \E_\lef$ is
        orthogonal.
    \item[\textbf{(D)}] The image $\Gamma_\lef$ is dense in $\E_\lef$.
    \item[\textbf{(I)}] The projection $\pi_\dwn|_{\Gamma}$ is injective,
    \item[\textbf{($\fd$-Reg)}] The window $\Omega_\lef$ is  regular, meaning
        that it is relatively compact and has non-empty interior, and its
        boundary is of measure $0$. We assume further that the $\fd$-codimensional
        Minkowski content of $\del\Omega_\lef$ in $\E_\lef$ is finite.
\end{itemize}
The condition \textbf{(O)} is simply stated 
for convenience; every cut and project set can be realised as a cut and project
set satisfying \textbf{(O)}, see for example the discussion in \cite[p. 9]{RSW}.
The conditions \textbf{(I)} and \textbf{(D)} below are standard assumptions in
the area (see, for example, \cite[Section 2]{RSW}), but we will prove as Lemma
\ref{prop:irrational} below that they are equivalent to the assumption of
$\Gamma$ being completely irrational with respect to $\E_\lef$, a formulation
which is more convenient for
our purposes. Finally, some condition on the regularity on the boundary of the
window is needed; otherwise, not even the patch counting function behaves as
expected \cite{JLO19}. Condition \textbf{($\fd$-Reg)} is the weakest condition
that allows for our techniques. Notice that, given these assumptions the inverse
$(\pi_\dwn)^{-1}: \Lambda \to \Gamma$ is well defined. For any subset $A\subset
\Lambda$, we set the notation $A^\up=(\pi_\dwn)^{-1}(A)$, and the notation $A
^\star=\pi_\lef((\pi_\dwn)^{-1}(A))$.  


The following lemmas tell us that
complete irrationality of the lattice with respect to $\E_\lef$ is equivalent to conditions \textbf{(D)} and
\textbf{(I)}. We note that even outside our current purposes, complete
irrationality is \emph{a priori} a much easier condition to check directly on a given lattice than
density of the projection on $\E_\lef$, say.

\begin{lemma} \label{lem:claim}
The lattice projection $\Gamma_\lef$ is dense in
    $\E_\lef$ if and only if for every $r > 0$, the intersection $\B(0,r) \cap
    \Gamma_\lef$ spans $\E_\lef$.
\end{lemma}
\begin{proof}
    We see immediately that density of $\Gamma_\lef$ in $\E_\lef$ implies the existence of
    arbitrarily small subsets of $\Gamma_\lef$ spanning $\E_\lef$; we now
    prove the other direction. We observe that $\pi_\lef : \Gamma
    \hookrightarrow \E_\lef$ is a group homomorphism. In particular, for every linearly independent
    $\Sigma \subset \Gamma_\lef$ of cardinality $d_\lef$, $\spn_\Z(\Sigma)$ is a lattice
    in $\E_\lef$. In order to prove our claim, it is sufficient to show that for every
    $x_\lef \in \E_\lef$ and $\eps > 0$, there exists $\Sigma \subset
    \Gamma_\lef$ such that $\dist(x_\lef,\spn_\Z(\Sigma)) < \eps$. We note that for
    any full rank lattice $\Theta \subset \E_\lef$,
    \begin{equation}
        \label{eq:distquotient}
        \sup_{x_\lef \in \E_\lef} \dist(x_\lef,\Theta) < \diam(\E_\lef/\Theta).
    \end{equation}
    Now, if $\Sigma$ is a basis for a lattice $\Theta$, then identifying
    $\E_\lef/\Theta$ with the parallelotope generated by $\Sigma$ we see
     that
    \begin{equation}
    \diam(\E_\lef/\Theta) < d_\lef \max\set{\abs \theta : \theta \in \Sigma}.
    \end{equation}
    In particular, by hypothesis for every $\eps > 0$ there exists $\Sigma_\eps
    \in \B_\lef(0,d_\lef^{-1}\eps) \cap \Gamma_\lef$ of cardinality $d_\lef$
    spanning $\E_\lef$, so that we deduce from \eqref{eq:distquotient} that
    $\dist(x,\spn_\Z(\Sigma_\eps)) < \eps$, and indeed existence of arbitrarily small subsets
    of $\Gamma_\lef$ spanning $\E_\lef$ implies density of $\Gamma_\lef$ in
    $\E_\lef$.
\end{proof}

\begin{lemma}\label{prop:irrational}
 The lattice $\Gamma$ is completely irrational with respect to $\E_\lef$,
 if and only if conditions $\textbf{(D)}$ and $\textbf{(I)}$ hold.
\end{lemma}

\begin{proof}
    The proof is split in two: we show that irrationality of $\Gamma$ with
    respect to $\E_\lef$ is equivalent to \textbf{(I)}, whereas irrationality of
    $\Gamma^\dagger$
    with respect to $\E_\lef$ is equivalent to \textbf{(D)}. The first assertion
    is direct: $\E_\lef = \ker(\pi_\dwn)$, so that we see directly
    that $\Gamma$ being irrational with respect to $\E_\lef$ is equivalent to
    $\pi_\dwn$ being injective when restricted to $\Gamma$.

    Now, suppose \textbf{(D)}, that $\Gamma_\lef$ is dense in
    $\E_\lef$. Then, for every $x_\lef \in \E_\lef\setminus\{0\}$, there exists $\gamma \in \Gamma$
    such that 
    \begin{equation}
    0 < \abs{x_\lef \cdot \gamma} = \abs{x_\lef \cdot \gamma_\lef} <
    1/2,
    \end{equation}
in particular $x_\lef \not \in \Gamma^\dagger$ and $\Gamma^\dagger$ is irrational with respect to $\E_\lef$. On the other hand, if
    $\Gamma_\lef$ is not dense in $\E_\lef$, use Lemma \ref{lem:claim} to find $r>0$ such that 
    \begin{equation}
        r = \inf \set{t : \spn(\B(0,t) \cap \Gamma_\lef) = \E_\lef}.
    \end{equation}
    Choose a codimension $1$ subspace $\F \subset \E_\lef$ spanned by elements of
    $\overline{\B(0,r)} \cap \Gamma_\lef$ and containing every $\gamma \in
    \B(0,r) \cap
    \Gamma_\lef$. (In particular for $d_\lef=1$ take $\F^\perp$ to be the whole $\E_\lef$.) Let $x \in \F^\perp \cap \E_\lef$ have norm $1$. By
    construction,
    \begin{equation}
        \set{y \in \E : x \cdot y \in (r/4,3r/4)} \cap \Gamma = \varnothing,
    \end{equation}
    so that Lemma \ref{lem:gammadaggersubspace} implies that $\spn(x)$ is a
    $\Gamma^\dagger$-subspace, in other words $\Gamma^\dagger \cap \E_\lef \ne
    \set{0}$.
\end{proof}

\begin{remark}
    We observe that complete irrationality of $\Gamma$ with respect to
    $\E_\lef$ does not preclude $\Lambda$ from containing periodic subsets
    of rank strictly smaller than $d$, so that these are also included in our
    analysis.
\end{remark}

We call a window satisfying Condition \textbf{($\fd$-Reg)} a $\fd$-regular
window. Boundedness of the window ensures that $\Lambda$ is uniformly discrete, and non-empty interior of the window guarantees that $\Lambda$  is relatively dense. The 
Minkowski content of $\del \Omega_\lef$ will determine how good of an upper
bound we are able to get on the discrepancy. 

\subsection{BL equivalence classes}

\label{sec:BLequiv}

Recall that we call a subset $Y$ of euclidean space a \emph{separated net} if
it is uniformly discrete and relatively dense. Two separated nets $Y, Y'$ are
said to be \emph{biLipschitz (BL) equivalent} if there is a biLipschitz map
$\phi: Y\to Y'$. We now describe sufficient conditions under which we can obtain
BL equivalence
of a separated net $Y$, \emph{a fortiori} of a cut and project set, to a lattice. We note
that since lattices are all in the same BL equivalence class, these
characterisations are often given in terms of equivalence to $\Z^{d_\dwn}$. 

Given a separated net $Y \subset \E_\dwn$, and a
search region $\Omega_\dwn$ and $\alpha > 0$, put
\begin{equation}
    \zeta_\alpha(\Omega_\dwn) =\max\set{\frac{\alpha \vol(\Omega_\dwn)}{\#(Y \cap
    \Omega_\dwn)},\frac{\#(Y \cap \Omega_\dwn)}{\alpha \vol(\Omega_\dwn)}}.
\end{equation}
Set $Q_\dwn \subset \E_\dwn$ to be the unit cube and for $t > 0$ put
\begin{equation}
    Z_\alpha(t) = \sup_{\bn \in \Z^{d_\dwn}} \zeta_\alpha(\bn + t Q_\dwn)
\end{equation}
The following lemma appears first as \cite{BK02} for $d_\dwn = 2$ and \cite[Theorem
3.1]{ACG13} for any $d_\dwn \ge 2$.
\begin{lemma}
    \label{lem:BLequivalent}
    Let $Y \subset \E_\dwn$ be a separated net, and suppose that there is
    $\alpha > 0$ such that
    \begin{equation}
        \sum_{n=1}^\infty \log Z_\alpha(2^n) < \infty.
    \end{equation}
    Then, $Y$ is BL-equivalent to $\Z^{d_\dwn}$.
\end{lemma}

We use this lemma as well as our Theorem \ref{thm:main} in order to prove
Theorem \ref{thm:mainBL}.

\noindent \textbf{Proof of Theorem \ref{thm:mainBL}.} 
Recall from Remark
\ref{rem:sinvariant} that it is equivalent to translate $\Omega_\dwn$ or
$\Gamma$ by $\bs_\dwn$. Therefore, since Theorem \ref{thm:main} is uniform in
the parameter $\bs$, we can read from \eqref{eq:withcubes} that for any
$\fd$-regular window $\Omega_\lef$, any lattice $\Gamma$
completely irrational with respect to $\E_\lef$ and so that $\Gamma^\dagger$ is
$\psi$-repelled by $\E_\lef$, any $\bs, \bs' \in \E$ and $\delta > 0$:
\begin{equation}
\frac{\#(\Lambda(\Gamma + \bs;\Omega_\lef) \cap (\bs'_\dwn + [0,t]^d))}{\alpha
t^d} = 1 + \bigo[\Gamma,\Omega_\lef]{\psi(t)^{-\fd + \delta \fd}} ,
\end{equation}
where we have set $\alpha = \frac{\vol(\Omega_\lef)}{\covol(\Gamma)}$. In
particular, if there is $\eta > 0$ so that $\psi(t) \gtrsim \log(t)^{\frac{1 +
\eta}{\fd}}$ then
\begin{equation}
    \label{eq:discrepconverges}
    \begin{aligned}
        \sum_{n=1}^\infty \log Z_\alpha(2^n) &\lesssim_\delta \sum_{n=1}^\infty
        \log\left(1 + \log(2^{-n})^{(- 1 + \delta)(1 + \eta) } \right) \\
        &\lesssim_\delta \sum_{n=1}^\infty n^{-1 - \eta + \delta + \eta\delta}.
\end{aligned}
\end{equation}
In particular, taking $\delta$ sufficiently small ensures that this is a
convergent sum. Therefore, by Lemma \ref{lem:BLequivalent} the cut and project set
$\Lambda(\Gamma+ \bs;\Omega_\lef)$ is BL equivalent to $\Z^d$.
\qed

\begin{remark}
    Suppose that $\psi(r)$ has a slower growth at infinity, that is that for every
    $\eta > 0$, as $r \searrow
    0$, $\psi(r) \lesssim \psi(r)^{\frac{1 + \eta}{s}}$, and that
    $\Gamma^\dagger$ is not $\psi'$-repelled by $\E_\lef$ for any $\psi'$ with a
    faster growth at infinity than this. Then, our methods
    cannot prove that the associated cut and project set is BL equivalent to a
    lattice from Lemma \ref{lem:BLequivalent}. Indeed, in that situation the sum
    at the end of \eqref{eq:discrepconverges} diverges. This means that from the
    point of view of our results on discrepancy this is the largest class of
    lattices which can be shown to be BL equivalent to a lattice; in order to
    improve it one needs to either find weaker sufficient conditions or
    improved results on the discrepancy. It is indicated in \cite{RSW} that the
    bound $t^{-d} \Delta(\Lambda(\Gamma + \bs;\Omega_\lef);\bn + [0,t]^d) =
    \smallo{1}$ cannot be improved to a different decay rate on the right-hand
    side that would hold for every cut and project set. In the appendix, Michael
    Bj\"orklund and Tobias Hartnick give variance estimates proving that our
    diophantine conditions is the weakest under which we can use Lemma
    \ref{lem:BLequivalent} to prove BL equivalence of a cut and project set to a
    lattice.
\end{remark}

\subsection{Acceptance domains}\label{sec:acceptance}

The goal of this section is to describe $r$-patterns in cut and project sets as
cut and project sets themselves. The upshot of this procedure is that any
theorem we prove for the statistics of cut and project sets is also valid for
statistics of $r$-patterns.

\begin{defi}
For $r > 0$, define
\begin{equation}
    \Gamma(r) := \set{\gamma \in \Gamma : \gamma_\dwn \in \B_\dwn(0,r)}.
\end{equation}
For every $P \subset \Gamma(r)$, the \emph{acceptance domain}
    $A_P$ is defined as
    \begin{equation}
        \label{eqdef:acceptance}
        A_P := \bigcap_{\gamma \in P} (\Omega_\lef - \gamma_\lef) \cap \bigcap_{\gamma \in
        \Gamma(r) \setminus P} (\inte((\Omega_\lef)^c) - \gamma_\lef),
    \end{equation}
    where the interior and the complement of $\Omega_\lef$ are taken relative to
    $\E_\lef$. 
 An \emph{$r$-admissible pattern} is a subset $P \subset \Gamma(r)$ containing the
    origin and such that $A_P \ne \varnothing$. 
   The set of all acceptance domains of size $r$ is denoted
    \begin{equation}
        \CA(r) := \set{A_P : P \text{ is a $r$-admissible pattern}}.
    \end{equation}
\end{defi}
\begin{remark}
    
    \begin{itemize}
        \item We note a difference in our definition with the one from
    \cite[Definition 3.3]{KoivusaloWaltonPolytopalI}: There patches $P$ are thought to be subsets of $\Lambda$, but here they are subsets of $\Gamma(r)$. 
    \item We take the interior of the complements rather than of
            the window as our windows are open whereas in  \cite[Definition 3.3]{KoivusaloWaltonPolytopalI} they are
            closed; the advantage of our choice is that it makes acceptance
            domains windows in their own right, see Lemma \ref{lem:apisawindow}.
        \item Another difference to  \cite[Definition 3.3]{KoivusaloWaltonPolytopalI} is that we consider in the intersections \emph{all} $\gamma \in
            \Gamma(r)$ rather than considering only the 'slab' of them so that
            $\gamma_\lef \in \Omega_\lef - \Omega_\lef$. This has no bearing on the
            definition: since $0 \in P$ for all $r$-admissible patterns,
            $\gamma_\lef \not \in \Omega_\lef - \Omega_\lef$ implies that
            $\Omega_\lef \cap (\Omega_\lef - \gamma_\lef) =
            \varnothing$. In other words, if such a $\gamma \in P$, then $A_P =
            \varnothing$. If such a $\gamma \not \in P$, $\Omega_\lef \subset \inte((\Omega_\lef) ^c) -
            \gamma_\lef$ so that intersecting with that set changes nothing. In
            the end, the intersections in \eqref{eqdef:acceptance} can always be
            supposed to be intersections over a finite set.
            \item In  \cite[Definition 3.3]{KoivusaloWaltonPolytopalI} the
                windows are assumed to be in a {\it generic} position (see \cite[Definition
            7.2]{BaakeGrimmBook}), so that $\del\Omega_\lef \cap \Gamma_\lef=\emptyset$. We will not make such an assumption. 
            \item For every $r$-admissible pattern $P$ the acceptance domain $A_P$ has non-empty interior since it
                is open and non-empty. As such, by the assumption {\bf(D)}, there exists $\lambda \in
            \Lambda$ such that $P_r(\lambda) = P_\dwn$.
    \end{itemize}
\end{remark}

\begin{lemma}
    \label{lem:apisawindow}
    Suppose that $\Omega_\lef$ is a $\fd$-regular window and let $r > 0$. Then for every
    $r$-admissible pattern $P \in
    \Gamma(r)$, $A_P$ is a $\fd$-regular window.
\end{lemma}

\begin{proof}
    The domain $A_P$ is a non-empty intersection of translates of $\Omega_\lef$ and of the interior
    of its complement. As such it is
    bounded, open, and its
    boundary is a union of boundary pieces for $\del \Omega_\lef$. Finiteness of
    the upper
    Minkowski content is stable under finite unions, so the boundary of $A_P$ satisfies the same
    condition as that of $\Omega_\lef$ and we conclude that $A_P$ is a $\fd$-regular
    window.
\end{proof}

The following lemma completely characterises pattern equivalence in terms of
maximal $r$-admissible patterns and their acceptance domains. It has appeared in the literature multiple times in various forms. For a proof we refer to \cite[Section 2]{KoivusaloWaltonPolytopalI}. 
\begin{lemma}
    For every $\lambda \in \Lambda$, 
    \begin{enumerate}
        \item $P_r(\lambda)^\up-\lambda^\up$ is the unique $r$-admissible pattern $P \subset
            \Gamma(r)$ such that $P_r(\lambda) = P_\dwn$.
        \item If $\zeta \in \Lambda$ is such that $\zeta^\star \in
            \label{property:patternequivalence}
            A_{P_r(\lambda)^\up}$, then $\zeta$ is $r$-pattern equivalent to
            $\lambda$. Conversely, if $\zeta$ is $r$-pattern equivalent to
            $\lambda$, then $\zeta^\star$ is in the closure of $A_{P_r(\lambda)^\up}$.
        
    \end{enumerate}
    \label{lem:patternscharacterisation}
\end{lemma}
\begin{remark}
Let us make a few remarks about the statement, as
    well as other definitions of acceptance domains in the literature.
    \begin{itemize}
        \item For a window $\Omega_\lef$ in a generic position, the statement
            \ref{property:patternequivalence} in the previous lemma does not
            need to refer to the closure, and becomes an equivalence. In either
            case, our results are the same \emph{whether or not some
            $\gamma_\lef$ hit $\del \Omega_\lef$}. 
        \item The previous lemma implies that our definition of acceptance
            domains makes them a tiling of $\Omega_\lef$. If, rather, we are interested in
            the frequency of \emph{$r$-clusters} of points in $\Lambda$, that is
            subsets of $\Lambda$ of diameter at most $2r$, the acceptance
            domains take the form of \eqref{eqdef:acceptance} without the
            intersection over translates of $\inte((\Omega_\lef)^c)$. In particular, this turns
            the relation ``belongs to an acceptance domain'' into a partial
            order (through reverse inclusion) rather than an equivalence, with
            that partial order expressing whether a cluster is a subset of
            another up to translation. This corresponds to the definition of
            acceptance domains found in, say \cite[Section
            7.2]{BaakeGrimmBook} and since acceptance domains are still
            $\fd$-regular windows (with the same proof as in Lemma
            \ref{lem:apisawindow}), our results on asymptotic frequency also
            apply to them.
    \end{itemize}
\end{remark}
We finish this section with the statement of a theorem on discrepancies for the
frequences of $r$-patterns. From the above discussion, Theorem
\ref{thm:frequency} follows directly from Theorem \ref{thm:main}. 
\begin{theorem}
    Under the hypotheses of Theorem \ref{thm:main}, let $\RM \subset \Lambda$ be
    the set of representatives for a $r$-pattern equivalence class. Then, there
    is a regular $\Omega_{\lef,\RM} \subset \Omega_\lef$ so that $\del
(\Omega_{\lef,{\RM}})$ has finite $(d_\lef-\fd)$-dimensional Minkowski content so
    that $\RM = \Lambda(\Gamma+\bs;\Omega_{\lef,{\RM}})$. In
    particular, for every $\delta > 0$ there is $C_{\RM,\fd, \delta}$ so that
    \begin{equation}
        \abs{\#(\RM\cap\B_\dwn(0,t)) -
    \frac{\vol(\B_\dwn(0,1))\vol(\Omega_{\lef,{\RM}})}{\covol(\Gamma)}} \le
    C_{\RM,\fd, \delta} t^{d_\dwn} \psi(t)^{-s + \delta s}
    \end{equation}
    \label{thm:frequency}
\end{theorem}
Of course, any improvement to Theorem \ref{thm:main} obtained when $\psi$ grows at
speed $\mu$ also applies here.

\section{Lattice point counts and the statement of the general theorems}
\label{sec:lattices}
In this section, we reduce Theorem \ref{thm:main} to a lattice point count for
$\Gamma$ in an anisotropically expanding set. Counting lattice points in ever
expanding sets is a standard problem in the geometry of numbers, we refer to the
methods and results in
\cite{KordyukovYakovlevZd,KordyukovYakovlevSpectralGeometry,lagace}, where lattice points
are counted in anisotropically expanding domains. The main difference in our
situation compared to these works is that the relevant
domains are products of domains in $\E_\lef$ and $\E_\dwn$; any expansion occurs
only along $\E_\dwn$. In particular, this allows us to consider sets whose
$\E_\lef$ slices have very bad boundary regularity.

In that spirit, we start with the following
general lemma which tells us that we can reduce the count of elements
in a cut and project set to a sum over the lattice $\Gamma$ in the total space.

\begin{lemma}
    \label{lem:sumrepresentation}
    Let $\Gamma \subset \E$ be a lattice which is irrational with
    respect to $\E_\lef$, $\bs \in \E$ and $\Omega_\lef \subset \E_\lef$, $\Omega_\dwn
    \subset \E_\dwn$. Then
    \begin{equation}
        \#\left(\Lambda(\Gamma+\bs;\Omega_\lef) \cap \Omega_\dwn\right)  = \sum_{\gamma \in \Gamma}
        \chi_{\Omega_\dwn \times \Omega_\lef}(\gamma - \bs).
    \end{equation}
\end{lemma}

\begin{proof}
    Irrationality of $\Gamma$ with respect to $\E_\lef$ implies that
    $\pi^\up : \Lambda \to (\Gamma + \bs) \cap
( \E_\dwn \times \Omega_\lef)$ is a bijection. By definition, $\lambda \in \Lambda \cap
    \Omega_\dwn$ if and only if $(\lambda^\up)_\dwn \in \Omega_\dwn$ and
    $(\lambda^\up)_\lef \in \Omega_\lef$; in other words $\lambda^\up \in
    \Omega_\dwn \times
    \Omega_\lef$. We deduce that
    \begin{equation}
        \#\set{\Lambda \cap \Omega_\dwn} = \#\set{\Lambda^\up \cap (\Omega_\dwn \times
        \Omega_\lef)} = \#\set{(\Gamma + \bs) \cap (\Omega_\dwn
        \times \Omega_\lef)} = \sum_{\gamma \in \Gamma}
        \chi_{\Omega_\dwn \times \Omega_\lef}(\gamma - \bs).
    \end{equation}
    This is our claim.
\end{proof}
\begin{remark}
    \label{rem:injectivity}
    We note that irrationality of $\Gamma$ with respect to $\E_\lef$ is only
    used to ensure that $\pi^\up$ is well defined (and in particular that
    $\pi_\dwn$ is injective). This can also be ensured from taking a window
    whose diameter is smaller than the element of $\Gamma \cap
    \E_\lef$ with smallest norm.
\end{remark}
In order to simplify the statement of theorems, we now set the notation, for
$\Omega = \Omega_\dwn \times \Omega_\lef$, $t > 0$, $\Gamma
\subset \E$ a lattice, and $\bs \in \E$,
\begin{equation}\label{eq:defcount}
    N(\Omega;\Gamma + \bs;t) := \#( (\Gamma + \bs) \cap (t\Omega_\dwn \times \Omega_\lef) =
        \sum_{\gamma \in \Gamma} \chi_{t\Omega_\dwn \times \Omega_\lef}(\gamma - \bs).
\end{equation}
We also slightly abuse notation and use the same symbol for the discrepancy in
the cut and project counting function and for the discrepancy in the lattice
point counting function, that is
\begin{equation}
    \label{eqdef:discreplattices}
    \Delta(\Omega;\Gamma+\bs;t) = N(\Omega;\Gamma+\bs;t) - \frac{\vol(\Omega)
    t^{d_\dwn}}{\covol(\Gamma)}.
\end{equation}
It follows from Lemma \ref{lem:sumrepresentation} that any estimates on the
discrepancy as defined in \eqref{eqdef:discreplattices} transfers to estimates
on the discrepancy for cut and project sets as defined in
\eqref{eqdef:discrepcandp}. The next theorems in this sections are
equivalent reinterprations of Theorems \ref{thm:main} and \ref{thm:lowerbound} in the
language of lattice point counting.
They are a refinement of  \cite[Theorems
2.1--2.3]{KordyukovYakovlevSpectralGeometry} and \cite[Theorem 1.11]{lagace}, where looking at product sets
allows us to be more precise in the dependence of the remainder estimates on the
regularity of the window, and with precise control on diophantine properties of
$\Gamma$, see Definition \ref{def:repel}.
\begin{theorem}
    \label{thm:maingeneral}
    Let  $\Gamma \subset \E$ be a lattice completely irrational with respect to
    $\E_\lef$ so that $\Gamma^\dagger$ is $\psi$-repellent with respect to $\E_\dwn\oplus\E_\lef$. Let $\Omega_\lef \subset \E_\lef$ be a $\fd$-regular
    window and $\Omega_\dwn \subset \E_\dwn$ be a convex set with
    $\RC^{\frac{d_\dwn-1}{2}}$ boundary and principal curvatures uniformly bounded
    away from $0$. If $\psi$ grows slowly at infinity, for every
    $\delta > 0$ there are  $C_{\Lambda, \delta}, t_0 > 0$, such that for every $\bs \in \E$ and
    every $t  > t_0$,
    \begin{equation}
        \abs{\Delta(\Omega;\Gamma + \bs;t)} \le C_{\Lambda, \delta} t^{d_\dwn} \psi(t)^{-s(1 - \delta)}.
    \end{equation}
    If $\psi$ grows at speed $\mu$ at infinity, then for every
    $\delta > 0$ there are $C,
    t_0 > 0$ such that for every $\bs \in \E$ and $t > t_0$,
    \begin{equation}
        \abs{\Delta(\Omega;\Gamma + \bs;t)} \le C t^{d_\dwn -
        \frac{2s d_\dwn}{(d_\dwn + 1)(s + \frac 1 \mu) + 2 d_\lef} + \delta}.
    \end{equation}
\end{theorem}

\begin{remark}
Theorem \ref{thm:main} is an immediate corollary of the first part of Theorem
\ref{thm:maingeneral} in the case where $\Omega_\dwn$ is the unit ball -- note that by the above considerations 
\[
|\#(\Lambda(\Gamma+\bs; \Omega_\lef)\cap t\Omega_\dwn )- \frac{\vol(\Omega_\dwn)\vol(\Omega_\lef) }{\covol(\Gamma)}t^{d_\dwn}|=|\Delta(\Omega; \Gamma+\bs;t)|. 
\]
Further, Theorem
\ref{thm:frequency} is also an immediate consequence, after observing that by
Lemma \ref{lem:apisawindow}, the representatives of an $r$-pattern form
cut and project set with a $\fd$-regular window. 
\end{remark}

Even when the boundary of the search region merely has finite perimeter we can
still obtain estimates on the discrepancy.

\begin{theorem}
    \label{thm:pwswindows}
    Let  $\Gamma \subset \E$ be a lattice completely irrational with respect to
    $\E_\lef$ so that $\Gamma^\dagger$ is $\psi$-repellent with respect to $\E_\dwn\oplus\E_\lef$. Let $\Omega_\lef \subset \E_\lef$ be a $\fd$-regular
    window and $\Omega_\dwn \subset \E_\dwn$ be a set with finite perimeter. 
    If $\psi$ grows slowly at infinity, for every
    $\delta > 0$ there are  $C, t_0 > 0$, such that for every $\bs \in \E$ and
    every $t  > t_0$,
    \begin{equation}
        \abs{\Delta(\Omega;\Gamma + \bs;t)} \le C t^{d_\dwn} \psi(t)^{-s(1 - \delta)}.
    \end{equation}
    If $\psi$ grows at speed $\mu$ at infinity, then for every
    $\delta > 0$ there are $C,
    t_0 > 0$ such that for every $\bs \in \E$ and $t > t_0$,
    \begin{equation}
        \abs{\Delta(\Omega;\Gamma + \bs;t)} \le C t^{d_\dwn -
        \frac{d_\dwn s}{d_\dwn s + d_\lef + \mu^{-1}} + \delta}.
    \end{equation}
\end{theorem}

\begin{remark}
    We first observe that when $\psi$ grows slowly at infinity, the
    estimate on the discrepancy does not depend on convexity properties of the search
    region. In fact, it depends only on the Diophantine properties of $\Gamma$
    and the regularity of the boundary of the window. 
\end{remark}

%
%
%

Theorems \ref{thm:maingeneral}, \ref{thm:pwswindows} (and hence Theorems \ref{thm:main} and \ref{thm:frequency}) are proved in Section \ref{sec:proofsupper}. 
The two previous theorems are uniform bounds on the discrepancy $\Delta$ with
respect to the parameter $\bs$. When we average over this parameter $\bs$ we
obtain lower bounds without conditions on the lattice $\Gamma$. We
note that when the dimension $d_\dwn \equiv 1 \mod 4$ there are arithmetic
obstructions which make our lower bounds slightly worse. This is similar to what is observed
in \cite{parnovskisobolevpolyharmonic,KonyaginSkriganovSobolev,lagaceparnovski},
and is not an artifact of the proof methods. We note that for these lower bounds
it is only required that the window be bounded with finite volume, no other
regularity assumptions is made. Even the boundedness assumption can be weakened,
see Remark \ref{rem:boundedtoostrong}

\begin{theorem}
    \label{thm:loweravg}
    Let $\Omega_\lef$ be any bounded window with non-zero volume, $\Omega_\dwn =
    \B_\dwn(0,1)$, $\Omega = \Omega_\dwn \times \Omega_\lef$ and $\Gamma \subset
    \E$ be any lattice. Then, there exists $A,C,t_0 > 0$ such that
    \begin{equation}
        \int_{\E/\Gamma} \Delta(\Omega,\Gamma + \bs,t) \de \bs = 0
    \end{equation}
    and for every $t > t_0$
    \begin{equation}
        \int_{\E/\Gamma} \abs{\Delta(\Omega,\Gamma + \bs, t)} \de \bs
        \ge C f(t) t^{\frac{d_\dwn-1}{2}},
    \end{equation}
    where $f : (0,\infty) \to (0,\infty)$ is defined as
    \begin{equation}
        f(t) := \begin{cases}
            1 & \text{if } d_\dwn \not \equiv 1 \mod 4 \\
            \exp(-A \log \log(t)^4) & \text{if } d_\dwn \equiv 1 \mod 4.
        \end{cases}
    \end{equation}
\end{theorem}

For the proof, see Section \ref{sec:average}. From this theorem we immediately obtain Theorem \ref{thm:lowerbound} and the following lower bounds on the best
possible upper bounds.

\begin{cor}
    Under the hypotheses of Theorem \ref{thm:loweravg}, there are $C,t_0 > 0$ such
    that for every $t > t_0$ there are two
    subsets $\CS_+(t)$ and $\CS_-(t)$ of $\E/\Gamma$ of positive measure so that 
    \begin{equation}
        \forall \bs \in \CS_\pm \qquad \qquad \pm
        \Delta(\Omega,\Gamma+\bs,t) \ge Cf(t) t^{\frac{d-1}{2}}.
    \end{equation}
\end{cor}

\section{Interlude I: approximation and mollification}\label{sec:smoothing}

\subsection{Approximation of sets and volume estimates}

In this section we will provide a smoothing of sets well-adapted to the products
of sets appearing in the study of pattern statistics. We use $\bullet \in
\set{\dwn,\lef}$ to
represent either direction since the construction itself is symmetric with
respect to the decomposition, but different choices end up being made in either
direction.

For $\Omega_\bullet \subset \E_\bullet$ and $a_\bullet > 0$ we define
\begin{equation}
    \Omega_{a_\bullet}^+  :=  \Omega_\bullet + \B_\bullet(0,a_\bullet) \qquad \text{and}
    \qquad \Omega_{a_\bullet}^- := \E_\bullet \setminus (\E_\bullet \setminus
    \Omega_\bullet)_{a_\bullet}^+.
\end{equation}
These are exterior and interior approximations to
$\Omega_\bullet$, respectively. We note that since $\B_\bullet(0,a_\bullet)$ is an open
convex set containing $0$,
\begin{equation}
    \label{eq:identitysumofsets}
    \Omega_{a_\bullet}^+ = \Omega_\bullet \cup 
    (\del
    \Omega_\bullet)_{a_\bullet}^+ 
    \qquad \text{and} \qquad \Omega_{a_\bullet}^- =
    \Omega_\bullet \setminus 
    (\del
    \Omega_\bullet)_{a_\bullet}^+.
\end{equation}
In particular we see that if $\Omega_\bullet$ has empty interior then
$\Omega_{a_\bullet}^-$ is empty. 
\begin{lemma}
    \label{lem:composition}
    For any $\bullet \in \set{\dwn,\lef}$, let $a_\bullet, t_\bullet > 0$
    and $\Omega \subset \E_\bullet$. We have the identity
    \begin{equation}
        (a_\bullet \Omega_\bullet)_{a_\bullet t_\bullet}^\pm =
        a_\bullet \Omega_{t_\bullet}^\pm.
    \end{equation}
\end{lemma}
\begin{proof}
    The equality for $+$ can be proven by observing that by linearity
    \begin{equation}
        a_\bullet(\Omega_{t_\bullet}^+) = a_\bullet(\Omega_\bullet +
        \B_\bullet(0,t_\bullet)) = a_\bullet \Omega_\bullet +
        \B_\bullet(0,a_\bullet t_\bullet) =
        (a_\bullet
        \Omega_\bullet)^+_{a_\bullet t_\bullet}.
    \end{equation}
    Similarly and using this first result we have that
    \begin{equation}
        a_\bullet\Omega_{t_\bullet}^- = \E_\bullet \setminus
        a_\bullet((\E_\bullet \setminus \Omega_\bullet)_{t_\bullet}^+) =
        \E_\bullet \setminus (\E_\bullet \setminus a_\bullet
        \Omega_\bullet)_{a_\bullet t_\bullet}^+ =
        (a_\bullet\Omega_\bullet)_{a_\bullet t_\bullet}^-.
    \end{equation}
\end{proof}
We need the following volume estimate.

\begin{lemma}
    \label{lem:relativevolestimateleft}
    \label{lem:volumeestimatedown}
    Let $\Omega_\bullet \subset \E_\bullet$ be a bounded set whose boundary has
    finite $\fd$-codimensional Minkowski content with respect to $\E_\bullet$, for
    some $0 \le \fd \le d_\bullet$. Then, there exists
    $C, \eps_0 > 0$ so that for every $0 < a_\bullet < \eps_0$ we have
    \begin{equation}
        \vol\left(\Omega_{a_\bullet}^+ \setminus \Omega_{a_\bullet}^-\right)
        \le C a_\bullet^\fd.
    \end{equation}
\end{lemma}
\begin{proof}
    It follows from identity \eqref{eq:identitysumofsets} that $\Omega_{
    a_\bullet}^+ \setminus \Omega_{ a_\bullet}^- = (\del \Omega_\bullet)^+_{
    a_\bullet}$, i.e. is the tubular neighbourhood of $\del \Omega_\bullet$ of radius
    $a_\bullet$. The claim then follows from the definition of Minkowski
    content.
\end{proof}

\subsection{Anisotropic mollifiers of indicators}

Now that we have found approximations of sets, it is time to find smooth
approximations to the indicator functions of those sets. 

Let $\tilde\rho \in \RC_c^\infty(\R)$ be a smooth nonnegative bump function supported
in $[-1,1]$. 
Define $\rho_\bullet \in \RC_c^\infty(\E_\bullet)$ as
\begin{equation}
    \rho_\bullet(\bx_\bullet) = \frac{\tilde\rho(|\bx_\bullet|)}{\int_{\E_\bullet}
    \rho(|\bx_\bullet|) \de \bx_\bullet},
\end{equation}
and $\rho \in \RC^\infty_c(\E)$ as
\begin{equation}
    \rho(\bx) = \rho_\dwn(\bx_\dwn) \rho_\lef(\bx_\lef).
\end{equation}

For $a_\bullet > 0$,
$\bullet \in \set{\dwn,\lef}$ define the bidisk
\begin{equation}
    \D(a_\dwn,a_\lef) := \B_\dwn(0,a_\dwn)\times \B_\lef(0,a_\lef) \subset \E,
\end{equation}
and
the mollifier $\rho^{(a_\dwn,a_\lef)}$
to be 
\begin{equation}
    \rho^{(a_\dwn,a_\lef)}(\bx) := \frac{1}{a_\dwn^{d_\dwn} a_\lef^{d_\lef}}
    \rho(a_\dwn^{-1}
    \bx_\dwn, a_\lef^{-1} \bx_\lef)
\end{equation}
We observe that $\rho^{(a_\dwn,a_\lef)}$ is supported in
$\D(a_\dwn,a_\lef)$, is smooth and has unit mass. By separation of variables,
its Fourier transform 
satisfies
\begin{equation}
    \begin{aligned}
        [\CF \rho^{(a_\dwn,a_\lef)}](\bxi) 
    &=  \left[ \CF_\dwn \rho_\dwn \right](a_\dwn \bxi_\dwn) \left[ \CF_\lef
        \rho_\lef
    \right](a_\dwn \bxi_\dwn)
\end{aligned}
\end{equation}
For any measurable $f : \E \to \R$ we define
\begin{equation}
    f^{(a_\dwn,a_\lef)} 
    := [f * \rho^{(a_\dwn,a_\lef)}](\bx) := \int_\E
    f(\bx - \by) \rho^{(a_\dwn,a_\lef)}(\by) \de \by.
\end{equation}
\begin{prop}
    \label{prop:upperlower}
    For $\bullet \in \set{\dwn,\lef}$, let $\Omega_\bullet \in \E_\bullet$ and
    $a_\bullet > 0$, and put $\Omega = \Omega_\dwn \times
    \Omega_\lef$. For all $\bx \in \E$,
    \begin{equation}
        \chi_{\Omega_{a_\dwn,a_\lef}^-}^{(a_\dwn,a_\lef)}(\bx) 
        \le
        \chi_\Omega(\bx) \le 
        \chi_{\Omega_{a_\dwn,a_\lef}^+}^{(a_\dwn,a_\lef)}(\bx).
    \end{equation}
\end{prop}
\begin{proof}
    We start with the rightmost inequality. Since $0 \le
    \chi_\Upsilon^{(a_\dwn,a_\lef)} \le 1$ for any $\Upsilon \subset \E$, it suffices to
    show that for every $\bx \in \Omega$,
    $\chi_{\Omega_{a_\dwn,a_\lef}^+}^{(a_\dwn,a_\lef)}(\bx) = 1$. But by
    definition of $\Omega^+_{a_\dwn,a_\lef}$, $\bx - \D(a_\dwn,a_\lef) \subset
    \Omega_{a_\dwn,a_\lef}^+$. Since $\supp(\rho_{a_\dwn,a_\lef}) \subset
    \D(a_\dwn,a_\lef)$ we deduce that
    \begin{equation}
        \begin{aligned}
            \chi_{\Omega_{a_\dwn,a_\lef}^+}^{(a_\dwn,a_\lef)}(\bx) &=
            \int_{\D(a_\dwn,a_\lef)}
            \chi_{\Omega_{a_\dwn,a_\lef}^+}(\bx - \by) \rho_{a_\dwn,a_\lef}(\by)
            \de \by
            \\&= \int_{\D(a_\dwn,a_\lef)}
            \rho_{a_\dwn,a_\lef}(\by) \de \by \\
        &=  1.
        \end{aligned}
    \end{equation}
    Similarly, to prove the leftmost inequality it suffices to show that for $\bx
    \in \E \setminus \Omega$, we have
    $\chi_{\Omega_{a_\dwn,a_\lef}^-}^{(a_\dwn,a_\lef)}(\bx) = 0$.
    This time, observe that for $\bullet \in \set{\dwn,\lef}$
    \begin{equation}
        \bx - \B_\bullet(0,a_\bullet) \subset \left(\E_\bullet \setminus \Omega_\bullet
        \right)_{a_\bullet}^+ = \E_\bullet \setminus \Omega_{a_\bullet}^-
    \end{equation}
    so that $\bx - \D(A_\dwn,A_\lef) \cap \Omega_{a_\dwn,a_\lef}^- = \varnothing$. This implies that
    \begin{equation}
        \chi_{\Omega_{a_\dwn,a_\lef}^-}^{(a_\dwn,a_\lef)}(\bx) =
        \int_{\D(a_\dwn,a_\lef)}
    \chi_{\Omega_{a_\dwn,a_\lef}}^-(\bx-\by) \rho_{a_\dwn,a_\lef}(\by) \de \by = 0,
    \end{equation}
    proving our claim.
\end{proof}

\section{Uniform upper bounds for the discrepancy}\label{sec:proofsupper}

We now aim at proving Theorem \ref{thm:maingeneral}. Let $\Omega_\lef \subset
\E_\lef$ be an open
set whose boundary has finite $\fd$-codimensional Minkowski content, and $\Omega_\dwn \subset
\E_\dwn$ be an open convex set satisfying the hypothesis of either Lemma
\ref{lem:Fourierdecay} or \ref{lem:badFourierdecay}, and put $\Omega = \Omega_\dwn \times \Omega_\lef$. We
suppose without loss of generality that the parameter $\bs = 0$; as mentioned in
Remark \ref{rem:sinvariant} this can be achieved by translating $\Omega$ instead
of $\Gamma$. For
$t > 0$, we put

\begin{equation}
    N(\Omega;\Gamma;t) = \sum_{\gamma \in \Gamma} \chi_{
    t\Omega_\dwn \times \Omega_\lef}(\gamma) = \sum_{\gamma \in \Gamma}
    \chi_{\Omega}(t^{-1}\gamma_\dwn, \gamma_\lef).
\end{equation}
Since our goal is to use the Poisson summation formula, we need to smooth out
this sum. For $a_\bullet > 0$, define
\begin{equation}
    \begin{aligned}
    N_{a_\dwn,a_\lef}^\pm(\Omega;\Gamma;t) &:= \sum_{\gamma \in \Gamma}
    \chi_{(t \Omega_\dwn \times
    \Omega_\lef)_{a_\dwn,a_\lef}^\pm}^{(a_\dwn,a_\lef)}(\gamma) \\
    &= \sum_{\gamma \in \Gamma} [\chi_{t \Omega^\pm_{t^{-1}a_\dwn} \times
    \Omega_{a_\lef}^\pm} * \rho^{(a_\dwn,a_\lef)}](\gamma) \\
    &= \sum_{\gamma \in \Gamma} [\chi_{\Omega^{\pm}_{t^{-1} a_\dwn} \times
    \Omega^{\pm}_{a_\lef}}*\rho^{(a_\dwn t^{-1},a_\lef)}](t^{-1}\gamma_\dwn,\gamma_\lef)     ,
\end{aligned}
\end{equation}
where we went from the first to the second line using Lemma
\ref{lem:composition}, and the second to third changing variables in the
integral defining convolution.
It follows from Proposition \ref{prop:upperlower} that
\begin{equation}
    N_{a_\dwn,a_\lef}^-(\Omega;\Gamma;t) \le N(\Omega;\Gamma;t) \le
    N_{a_\dwn,a_\lef}^+(\Omega;\Gamma;t)
\end{equation}
as such we now aim to find an asymptotic expression for
$N_{a_\dwn,a_\lef}^\pm(\Omega;\Gamma;t)$. We will
choose $a_\dwn$ and $a_\lef$ at the end of the process, as functions of $t$, the dimensions
$d_\dwn$ and $d_\lef$, as well as the parameter $\fd$. Recall from Section
\ref{sec:latticesdef} that $\Gamma^\dagger$ is the dual lattice of $\Gamma$,
that is
the lattice
\begin{equation}
    \Gamma^\dagger := \set{\gamma^\dagger \in \E : \gamma^\dagger \cdot \gamma
    \in \Z \text{ for all } \gamma \in \Gamma}.
\end{equation}
From the Poisson summation formula, we have that
\begin{equation}
    \begin{aligned}
        N_{a_\dwn,a_\lef}^\pm(\Omega;\Gamma,t) &=
        \frac{1}{\covol(\Gamma)}\sum_{\gamma^\dagger \in \Gamma^\dagger}
        [\CF \chi_{(t\Omega_\dwn \times
        \Omega_\lef)_{a_\dwn,a_\lef}^\pm}^{(a_\dwn,a_\lef)}](\gamma^\dagger) \\
        &=  
        \frac{t^{d_\dwn}}{\covol(\Gamma)} \sum_{\gamma^\dagger \in \Gamma^\dagger}
        \left[ \CF_\dwn \chi_{\Omega_{t^{-1}A_\dwn}^\pm}
        \right](t\gamma^\dagger_\dwn)
        \left[ \CF_\lef \chi_{\Omega_{a_\lef}^\pm} \right](
            \gamma^\dagger_\lef) 
            \left[ \CF_\dwn \rho_\dwn \right](a_\dwn \gamma^\dagger_\dwn)
                \left[ \CF_\lef \rho_\lef \right](a_\lef \gamma^\dagger_\lef).
    \end{aligned}
\end{equation}
This last sum can be further split into two terms, the term at $0$ and the remainder:
\begin{equation}
    N_{a_\dwn,a_\lef}^\pm(\Omega;\Gamma;t) = V(\Omega;\Gamma;t;a_\dwn,a_\lef) +
    R(\Omega;\Gamma;t;a_\dwn,a_\lef)
\end{equation}
Here $V$, the volume term, is the term at $0$ in the sum:
\begin{equation}
    \label{eq:volterm}
V(\Omega;\Gamma;t;a_\dwn,a_\lef) =
\frac{t^{d_\dwn}}{\covol(\Gamma)} \left[ \CF_\dwn \chi_{\Omega_{t^{-1}a_\dwn}^\pm}
        \right](0)
        \left[ \CF_\lef \chi_{\Omega_{a_\lef}^\pm} \right](0) 
            \left[ \CF_\dwn \rho_\dwn \right](0)
            \left[ \CF_\lef \rho_\lef \right](0).
\end{equation}
On the other hand $R$, the remainder term, is the sum over all other elements of
the dual lattice:
\begin{equation}
    R(\Omega;\Gamma;t;a_\dwn;a_\lef) = 
    \frac{t^{d_\dwn}}{\covol(\Gamma)} \sum_{\gamma^\dagger \in \Gamma^\dagger
        \setminus \set 0}
        \left[ \CF_\dwn \chi_{\Omega_{t^{-1}a_\dwn}^\pm}
        \right](t\gamma^\dagger_\dwn)
        \left[ \CF_\lef \chi_{\Omega_{a_\lef}^\pm} \right](
            \gamma^\dagger_\lef) 
            \left[ \CF_\dwn \rho_\dwn \right](a_\dwn \gamma^\dagger_\dwn)
                \left[ \CF_\lef \rho_\lef \right](a_\lef \gamma^\dagger_\lef).
\end{equation}
\subsection{The volume term}

We first analyse $V$. The mollifiers $\rho_{\bullet}$ were constructed
specifically to have unit mass so that $[\CF_\bullet \rho_\bullet](0) = 1$.
On
to the Fourier transforms of indicators, we have that for any $s_\bullet > 0$
\begin{equation}
    [\CF_\bullet \chi_{\Omega_{s_\bullet}^\pm}](0) =
    \Vol(\Omega_{s_\bullet}^\pm).
\end{equation}
Since 
\begin{equation}
    \Omega_{s_\bullet}^- \subset
    \Omega_\bullet\subset
    \Omega_{s_\bullet}^+
\end{equation}
we deduce that
\begin{equation}
    \abs{\vol(\Omega_{s_\bullet}^\pm) - \vol(\Omega_{\bullet})} \le
    \vol(
    \Omega_{s_\bullet}^+ \setminus 
\Omega_{s_\bullet}^-),
\end{equation}
and inserting into \eqref{eq:volterm} with $t^{-1}a_\dwn$ and $a_\lef$ in place
of $s_\dwn$ and $s_\lef$ gives us
\begin{equation}
    \label{eq:volumediscrep}
    \begin{aligned}
   \abs{ V(\Omega;\Gamma;t;a_\dwn,a_\lef) -
   \frac{t^{d_\dwn}\vol(\Omega)}{\covol(\Gamma)}} &\le \bigg(\vol(\Omega_\dwn)
   \vol(\Omega_{a_\lef}^+ \setminus \Omega_{a_\lef}^-) + \vol(\Omega_\lef)
   \vol(\Omega_{t^{-1}a_\dwn}^+ \setminus \Omega_{t^{-1}a_\dwn}^-)\\&\qquad +   
   \vol(\Omega_{t^{-1}a_\dwn}^+ \setminus \Omega_{t^{-1}a_\dwn}^-)
   \vol(\Omega_{a_\lef}^+ \setminus
   \Omega_{a_\lef}^-)\bigg)\frac{t^{d_\dwn}\vol(\Omega)}{\covol(\Gamma)}.
   \end{aligned}
\end{equation}
Since we end up choosing $a_\dwn = \smallo{t}$ and $a_\lef = \smallo{1}$ in such a way that
$$\vol(\Omega_{a_\lef}^+ \setminus \Omega_{a_\lef}^-) \to 0 \qquad \text{and}
\qquad \vol(\Omega_{t^{-1}a_\dwn}^+ \setminus \Omega_{t^{-1} a_\dwn}^-) \to
0$$ as
$t \to \infty$, it is sufficient to
bound only the first two terms in \eqref{eq:volumediscrep}. 
Lemma \ref{lem:volumeestimatedown} tells us that as $t^{-1} a_\dwn \to 0$ and
$a_\lef \to 0$,
\begin{equation}
    \label{eq:smallleft}
    \vol(\Omega_{t^{-1}a_\dwn}^+ \setminus \Omega_{t^{-1}a_\dwn}^-) 
    \lesssim t^{-1} a_\dwn \qquad \text{and} \qquad 
    \vol(\Omega_{a_\lef}^+ \setminus \Omega_{a_\lef}^-)
    \lesssim a_\lef^\fd .
\end{equation}
Putting together the last two displays we deduce that
\begin{equation}
    \label{eq:errorvolume}
    \abs{V(\Omega;\Gamma;t;a_\dwn;a_\lef) -
        \frac{t^{d_\dwn} \Vol(\Omega)}{\covol(\Gamma)}} \lesssim_{\Omega,\Gamma} 
        t^{d_\dwn}\left(a_\lef^{\fd} \vol(\Omega_\dwn)+ 
    t^{-1}a_\dwn \vol(\Omega_\lef) \right)
\end{equation}

\subsection{The remainder term}

Since
$\rho_\bullet$ is smooth $\CF_\bullet
\rho_\bullet$ is a Schwartz function. That is, for every $K_\bullet >0$
\begin{equation}
    \abs{[\CF_\bullet \rho_\bullet](a_\bullet \bxi_\bullet)} \lesssim_{K_\bullet} \left(1
    + \abs{a_\bullet \bxi_\bullet}\right)^{-K_\bullet}.
\end{equation}
For $\CF_\bullet \chi_{\Omega_{a_\bullet}^\pm}$ we use the estimate, uniform in
$\bxi_\bullet$
\begin{equation}
    \label{eq:coarseest}
    \abs{\left[
    \CF_\bullet \chi_{\Omega^\pm_{a_\bullet}}\right](\bxi_\lef)} \lesssim_{\Omega_\lef} (1 +
    \abs{\bxi_\bullet})^{-L_\bullet}
\end{equation}
where
\begin{equation}
    \label{eq:preciseest}
    L_\bullet = \begin{cases}
        \frac{d_\bullet+1}{2} & \text{if $\Omega_\bullet$ is strictly convex
        with principal curvatures bounded away from zero}; \\
   1 & \text{if $\Omega_\bullet$ has finite perimeter}; \\
   s & \text{if $\Omega_\bullet$ is $s$-regular, $0 < s < 1$},
   \end{cases}
\end{equation}
Combining \eqref{eq:coarseest} and
\eqref{eq:preciseest}, we have for every $K_\bullet > d_\bullet -
L_\bullet$ (so that the sum converges)
\begin{equation}
    \label{eq:errorremainder}
    \begin{aligned}
        \abs{R(\Omega;\Gamma;t;a_\dwn,a_\lef)} &\lesssim \frac{t^{d_\dwn}}{\covol(\Gamma)} \sum_{\gamma^\dagger
            \in \Gamma^\dagger \setminus \set 0 } 
            \frac{
                (1 +
            |t\gamma^\dagger_\dwn|)^{-L_\dwn} (1 + |
        \gamma^\dagger_\lef|)^{-L_\lef} }{\left( 1 + |a_\dwn
                \gamma_\dwn^\dagger|\right)^{K_\dwn} \left(1 + |a_\lef
            \gamma_\lef^\dagger|\right)^{K_\lef}} \\
            &\lesssim_\Gamma t^{d_\dwn} \left(  \underbrace{\sum_{0 <
                    |\gamma^\dagger_\dwn| < t^{-\sigma}} \frac{(1 +
            |\gamma_\lef^\dagger|)^{-L_\lef}}{(1 + |a_\lef
        \gamma_\lef^\dagger|)^{K_\lef}}}_{ \Sigma_1 :=} + 
\underbrace{\sum_{|\gamma_\dwn^\dagger| > t^{-\sigma}}  
            \frac{
                (1 +
            |t\gamma^\dagger_\dwn|)^{-L_\dwn} (1 + |
        \gamma^\dagger_\lef|)^{-L_\lef} }{\left( 1 + |a_\dwn
                \gamma_\dwn^\dagger|\right)^{K_\dwn} \left(1 + |a_\lef
        \gamma_\lef^\dagger|\right)^{K_\lef}}}_{ \Sigma_2:=} 
    \right) ,
\end{aligned}
\end{equation}
where $\sigma > 0$ is a parameter to determine later.

\textbf{First region: small $\gamma_\dwn$.}
To estimate $\Sigma_1$, we first use Peetre's inequality to see that for every
$\bxi \in \E/\Gamma^\dagger$,
\begin{equation}
    \begin{aligned}
    \Sigma_1 &\le \sum_{0 < |\gamma_\dwn^\dagger| < t^{-\sigma}} 
    \frac{(1 +
    |\gamma_\lef^\dagger + \xi_\lef|)^{-L_\lef}}{(1 + |a_\lef
    (\gamma^\dagger_\lef + \xi_\lef))^{K_\lef}} 
    \frac{(1 +
    |\xi_\lef|)^{L_\lef}}{(1 + |\xi_\lef|)^{-K_\lef}} \\
    &\le (1 + \diam(\E/\Gamma^\dagger))^{L_\lef + K_\lef} \sum_{0 <
    |\gamma_\dwn^\dagger| < t^{-\sigma}} 
    \frac{(1 +
    |\gamma_\lef^\dagger + \xi_\lef|)^{-L_\lef}}{(1 + |a_\lef
    (\gamma^\dagger_\lef + \xi_\lef)|)^{K_\lef}}.
        \end{aligned}
\end{equation}
Integrating both sides over $\E/\Gamma^\dagger$ therefore gives us
\begin{equation}
    \Sigma_1 \lesssim_\Gamma \sum_{0 < |\gamma_\dwn^\dagger| < t^{-\sigma}}
\int_{\E/\Gamma^\dagger} \frac{(1 + |\gamma_\lef^\dagger + \xi_\lef|)^{-
L_\lef}}{(1 + |a_\lef(\gamma_\lef^\dagger + \xi_\lef|)^{K_\lef}} \de \bxi.
\end{equation}
Since $\Gamma^\dagger$ is $\psi$-repulsive, in
this region $|\gamma_\lef^\dagger| > \psi(t^{\sigma})$, and in particular since
$\E/\Gamma^\dagger$ is bounded we see that $|\gamma_\lef^\dagger + \xi_\lef| >
\psi(t^{\sigma})/2$ as soon as $t$ is large enough, whereas in this region
$\gamma_\dwn + \xi_\dwn$ is merely uniformly bounded. Therefore, we can unfold
the previous sum in that region to obtain
\begin{equation}
    \begin{aligned}
    \Sigma_1 &\lesssim_\Gamma \int_{|\xi_\dwn| < 2 \diam(E/\Gamma^\dagger)}
    \int_{|\xi_\lef| > \psi(t^{\sigma})/2}  (1 + |\xi_\lef|)^{-L_\lef} (1 +
    a_\lef|\xi_\lef|)^{-K_\lef} \de \xi_\lef \de \xi_\dwn \\ 
    &\lesssim_{\Gamma} a_\lef^{-K_\lef} \int_{|\xi_\lef| > \psi(t^{\sigma})/2}
    \frac{\de \xi_\lef}{\abs{\xi_\lef}^{L_\lef + K_\lef}} \\
    &\lesssim_{\Gamma} a_\lef^{-K_\lef} \psi(t^{\sigma})^{d_\lef-L_\lef - K_\lef}.
\end{aligned} 
\end{equation}

\textbf{Second region: larger $\gamma_\dwn$.} In this region, we estimate $(1 +
|\gamma_\lef^\dagger|)^{-L_\lef} \lesssim 1$, and use Peetre's inequality,
assuming that we choose $a_\bullet \le 1$ small, to get 
\begin{equation}
\begin{aligned}
        \label{eq:peetreII}
        \left(1 + |a_\bullet \gamma^\dagger_\bullet|\right)^{-K_\bullet}
        &\le
        \left(1+|a_\bullet(\gamma_\bullet^\dagger+\bxi_\bullet)|\right)^{-K_\bullet}
        \left(1+|a_\bullet\bxi_\bullet|\right)^{K_\bullet}\\
        &\lesssim_{K,\Gamma} \left(1 + |a_\bullet(\gamma^\dagger_\bullet +
        \bxi_\bullet)| \right)^{-K_\bullet}.
        \end{aligned}
    \end{equation}
    We refine Peetre's inequality on the
    region
    $|\gamma_\dwn^\dagger| > t^{-\sigma}$. Observe that
    \begin{equation}
        \label{eq:refinedpeetre}
        \begin{aligned}
                    \frac{1 + |t(\gamma_\dwn+ \xi_\dwn)|}{1 + t|\gamma_\dwn|} &\le \frac{1 +
        t(|\gamma_\dwn| + |\xi_\dwn|)}{1 + t|\gamma_\dwn|} \\
        &\le \frac{1 + t(|\gamma_\dwn| + \diam(\E/\Gamma^\dagger))}{1 +
        t|\gamma_\dwn|}.
    \end{aligned}
    \end{equation}
    It is a straightforward calculus exercise to see that for every $t, c > 0$,
    the function
    \begin{equation}
        x \mapsto \frac{1 + t(x + c)}{1 + tx}
    \end{equation}
    is strictly decreasing on $(0,\infty)$, so that the right-hand side in
    \eqref{eq:refinedpeetre} is bounded by the value at $|\gamma_\dwn| =
    t^{-\sigma}$. Evaluating gives us
    \begin{equation}
        (1 + t|\gamma_\dwn|)^{-L_\dwn} \lesssim_{L_\dwn,\Gamma} t^{L_\dwn \sigma} (1 + t|\gamma_\dwn
        + \xi_\dwn|)^{-L_\dwn}.
    \end{equation}
    Putting all these estimates back into $\Sigma_2$, integrating over
    $\E/\Gamma^\dagger$ then unfolding the sum gives us
    \begin{equation}
        \begin{aligned}
        \Sigma_2 &\lesssim_{K_\bullet,\Gamma,L_\bullet}
        t^{L_\dwn \sigma} \sum_{|\gamma_\dwn^\dagger| > t^{-\sigma}} \int_{\E/\Gamma^\dagger}
        \frac{(1 +t|\gamma_\dwn^\dagger + \xi_\dwn|)^{-L_\dwn}}{(1 +
            a_\dwn|\gamma_\dwn^\dagger + \xi_\dwn|)^{K_\dwn}(1 +
        a_\lef|\gamma_\lef^\dagger + \xi_\lef|)^{K_\lef}} \de \bxi \\
        &\lesssim_{K_\bullet,\Gamma,L_\bullet} t^{L_\dwn \sigma} \int_\E \frac{(1
        + t |\xi_\dwn|)^{-L_\dwn}}{(1 + a_\dwn|\xi_\dwn|)^{K_\dwn}(1 +
        a_\lef|\xi_\lef|)^{K_\lef}} \de \bxi.
 \end{aligned}
    \end{equation}
    Changing variables as $(\xi_\dwn, \xi_\lef) \mapsto (a_\dwn^{-1} \xi_\dwn,
    a_\lef^{-1} \xi_\lef)$ gives us 
    \begin{equation}
        \begin{aligned}
            \Sigma_2 &\lesssim_{K_\bullet, \Gamma,L_\bullet} \frac{t^{-(1 -
            \sigma) L_\dwn} a_\dwn^{L_\dwn}}{a_\dwn^{d_\dwn} a_\lef^{d_\lef}} \int_{\E}
            \frac{|\xi_\dwn|^{-L_\dwn}}{(1 + |\xi_\dwn|)^{K_\dwn}(1 +
            |\xi_\lef|)^{K_\lef}} \de \bxi \\
            &\lesssim_{K_\bullet,\Gamma,L_\bullet} t^{-(1-\sigma)L_\dwn}
            a_\dwn^{L_\dwn - d_\dwn} a_\lef^{-d_\lef},
    \end{aligned}
    \end{equation}
    as long as $K_\lef > k-d$, and we take $K_\dwn$ arbitrary large
    (since it plays no role in the asymptotics appart from the constant in
    front). 

    \textbf{Combining the remainder terms.} Summing up $\Sigma_1$ and $\Sigma_2$
    we obtain in the end that
    \begin{equation}
        |R(\Omega,\Gamma,t,a_\dwn,a_\lef)| \lesssim_{\Omega,\Gamma,
        K_\bullet,L_\bullet} t^{d_\dwn} a_\lef^{-K_\lef}
        \psi(t^{\sigma})^{d_\lef -
        L_\lef - K_\lef} + t^{d_\dwn - (1 - \sigma) L_\dwn} a_\dwn^{L_\dwn -
        d_\dwn}
    a_\lef^{-d_\lef}.
    \end{equation}

\subsection{A balancing act}
Along the proof, we have introduced arbitrary parameters $a_\dwn, a_\lef,
K_\lef$ and $\sigma$; it is now time to choose them carefully in order to obtain
the required bounds on the discrepancy. Note that $a_\bullet$ will be functions
of $t$, which is our asymptotic parameter; as such any constants in asymptotic
estimates is not allowed to depend on them. On the other hand, $\sigma$ and
$K_\lef$ will depend only on $\psi$ which is part of the geometry of the
cut and project set, so our asymptotic estimates may depend without issue from
them. Technically, $K_\dwn$ is also arbitrary but it plays no
role in the asymptotics appart from needing to be large enough for some integral
to converge. Summing up the contributions from $\Sigma_1, \Sigma_2$ and the
remainder term and factoring $t^{d_\dwn}$ out we arrive to
\begin{equation}
    \label{eq:tobalance}
t^{-d_\dwn}\Delta(\Gamma;t;\Omega) \lesssim_{\Omega,\Gamma,K_\lef}
\, \, a_\lef^s + t^{-1} a_\dwn + a_\lef^{-K_\lef}
    \psi(t^{\sigma})^{d_\lef-L_\lef-K_\lef} + t^{-(1-\sigma) L_\dwn}
    a_\dwn^{L_\dwn - d_\dwn} a_\lef^{-d_\lef}.
\end{equation}
Since $d_\lef - L_\lef > 0$ and $\psi(t^{\sigma }) \to \infty$ as $t \to
\infty$, we see immediately that if we choose $a_\lef(t) \lesssim
\psi(t^{\sigma})^{-1}$ the third term  in \eqref{eq:tobalance} is unbounded. As
such, we choose $a_\lef(t) = \psi(t^{\sigma})^{-1+\delta}$ for some fixed
$0 < \delta < 1$ (recalling that we assumed in our earlier analysis that
$a_\lef$ should be sufficiently small). Once $\delta$ is fixed, taking $K_\lef$ 
arbitrarily large will make the the third term smaller than the first. From
these choices we are now left with
\begin{equation}
    \label{eq:nowtobalance}
    t^{-d_\dwn} \Delta(\Gamma;t;\Omega) \lesssim_{\Omega,\Gamma,K_\lef,\delta}
    \psi(t^{\sigma})^{-s + \delta s} + t^{-1} a_\dwn + 
    t^{-(1-\sigma)L_\dwn}
    \psi(t^{\sigma})^{-d_\lef(-1 + \delta) } a_\dwn^{L_\dwn - d_\dwn}.
\end{equation}
There are now two situations, depending on the growth of $\psi$ at infinity.

\textbf{Case 1: $\psi$ grows slowly at infinity}

In this situation $\psi(t^{\sigma}) \lesssim_\eps t^{\eps}$
for all $\eps> 0$. Here, the first term on
the right-hand side of \eqref{eq:nowtobalance} would preclude any polynomial
bound on $t^{-d} \Delta$, but choosing $a_\dwn = t^{1/2}$ and $\sigma = 1$, we
see that the second and third term are bounded above by a negative power of $t$.
Therefore, in this situation we obtain
\begin{equation}
    t^{-d_\dwn} \Delta(\Gamma;t;\Omega) \lesssim_{\Omega,\Gamma,K_\lef,\delta}
    \psi(t)^{-s + \delta s}
\end{equation}
for any $\delta > 0$.

\textbf{Case 2: $\psi$ grows at speed $\mu$ at infinity.}

In this case $\psi(t^{\sigma})\asymp t^{\sigma\mu}$ and we can rewrite \eqref{eq:nowtobalance} as
\begin{equation}
    \label{eq:newnowtobalance}
    t^{-d_\dwn} \Delta(\Gamma;t;\Omega) \lesssim_{\Omega,\Gamma,K_\lef,\delta}
    t^{-\sigma \mu s + \delta \sigma \mu s} + t^{-1} a_\dwn +
    t^{-(1-\sigma)L_\dwn - d_\lef \sigma\mu(-1 + \delta)} 
     a_\dwn^{L_\dwn - d_\dwn}.
\end{equation}

From this we see that the discrepancy will be made as small as possible if we
choose the parameters $0 < \sigma < 1$, $0 < a_\dwn \ll t$, and $0 < \delta < 1$ in such a way that
\begin{equation}
    t^{\sigma \mu (-s + \delta s)} = t^{-1} a_\dwn = 
    t^{-(1-\sigma)L_\dwn - d_\lef \sigma \mu(-1 + \delta)}
     a_\dwn^{L_\dwn - d_\dwn}.
\end{equation}
Indeed, the first two terms are independent of each other and have opposite
monotonicity with the third term in every parameter. This leads directly to
taking $a_\dwn = t^{1 + \sigma \mu (-s + \delta s)}$.
Replacing in the third term and equating with the first we choose 
\begin{equation}
    \sigma = \frac{d_\dwn}{s \mu (1 - \delta) (d_\dwn - L_\dwn + 1) + L_\dwn + (1 -
    \delta) d_\lef \mu}
\end{equation}
so that
\begin{equation}
    t^{-d_\dwn} \Delta(\Gamma;t;\Omega) \lesssim_{\Omega,\Gamma,\delta} \, \,
    t^{\frac{-d_\dwn
        (1 - \delta)}{(1 - \delta) (d_\dwn - L_\dwn + 1 + \frac{d_\lef}{s}) +
    \frac{L_\dwn}{\mu s} }}
\end{equation}
We see that this exponent is smaller when
$\delta$ approaches $0$, which means that in the end, for any $\delta > 0$ we
get
\begin{equation}
    t^{-d_\dwn} \Delta(\Gamma;t;\Omega) \lesssim_{\Omega,\Gamma,\delta} \, \,
    t^{\frac{-d_\dwn}{d_\dwn - L_\dwn + 1 + \frac{d_\lef}{s} +
    \frac{L_\dwn}{\mu s} } + \delta}.
\end{equation}
\qed

\section{Interlude II: Diophantine properties of lattices}\label{sec:quantitative}
In order to obtain averaged bounds on the discrepancy in Theorem
\ref{thm:loweravg}, we need to get quantitative lower bounds on the sizes of
lattice vector projections. Notice that we do not make any irrationality assumptions on
$\Gamma$ in this section. 
\begin{lemma}
    \label{lem:pigeons}
Let $\Gamma \subset \E$ be a lattice and suppose that $d_\dwn \ge 2$. There exist $C,t_0 > 0$ such
that for every $t > t_0$ there are linearly independent $\gamma_1, \gamma_2 \in
\Gamma$ such that
    \begin{itemize}
        \item for $j \in \set{1,2}$, $t \le |\gamma_{j,\dwn}| \le 5t$;
        \item the angle between $\gamma_{1,\dwn}$ and $\gamma_{2,\dwn}$ is at
            least $\pi/3$.
        \item for $j \in \set{1,2}$, $|\gamma_{j,\lef}| \le C t^{\frac{-d_\dwn}{d_\lef}}$.
    \end{itemize}
\end{lemma}

\begin{proof}
    Suppose first that $\E_\lef$ is a $\Gamma^\dagger$-subspace. Then, $\E_\dwn$
    is a $\Gamma$-subspace and we can use standard lattice point counting
    results in $\E_\dwn$ to see that there exists at least two linearly
    independent vectors in $\Gamma \cap (\B_\dwn(3t) \setminus \B_\dwn(t)$ for
        every $t$ large enough. 

        Suppose now that $\E_\lef$ is
    not a $\Gamma^\dagger$-subspace, so that the $\Gamma$-subspace
    $\F := \Gamma^\dagger(\E_\lef)^\perp$ strictly contains $\E_\dwn$. Then $\Theta
    := \Gamma \cap \F$ is a sublattice of $\Gamma$ spanning $\F$. Note that it
    is possible that $\Theta^\dagger$ (when viewed as a lattice in
    $\F$) intersects $\F_\lef$, but repeating the previous argument reduces the
    dimension of the left subspace everytime, so either we are left with
    $\E_\dwn$ eventually a $\Gamma$-subspace, in which case the original lattice
    point counting argument apply, or eventually $\Theta^\dagger$ is irrational with
    respect to $\F_\lef$, which we now assume,

    Put $s = 1$ if $\Gamma \cap \F = \varnothing$, or, following Remark
    \ref{rem:injectivity} put $s = \min\set{|\gamma|: 0 \ne \gamma \in \Gamma
    \cap \E_\lef}$, so that $\pi_\dwn$ is injective when restricted to $\Theta
    \cap (\E_\dwn \times \B_\F(0,s))$. Then,
    $\Lambda(\F,\E_\dwn;\F_\lef;\Theta;\B_{\F_\lef}(0,s))$ is a cut and
    project set. Let $\bu_1, \bu_2 \in \E_\dwn$ be orthogonal unit vectors. For $j
    \in \set{1,2}$, let
    \begin{equation}
        \Omega_{j,\dwn} := \set{\bx_\dwn \in \E_\dwn : \frac{|\bx_\dwn \cdot
        \bu_j|}{|\bx_\dwn|} \le \frac{\pi}{12}, 0 < |\bx_\dwn| \le 1}.
    \end{equation}
    Then, for any $t > 0$, if $\bx_1 \in t \Omega_{1,\dwn}$ and $\bx_2 \in t
    \Omega_{2,\dwn}$, then $\bx_1$ and $\bx_2$ have angle at least $\pi/3$. By
    Theorem \ref{thm:main}, we have that for $j \in \set{1,2}$, $t > 0$ and $n
    \in \N$,
    \begin{equation}
        \#(\Lambda \cap (n+1)t \Omega_{j,\dwn}) - \# \Lambda \cap nt\Omega_{j,\dwn}
        \gtrsim t^{d_\dwn}
    \end{equation}
    By the pigeonhole principle, this means that there are at least
    two points $\lambda_j \in \Lambda
    \cap (2t \Omega_{j,\dwn} \setminus t \Omega_{j,\dwn})$ and $\zeta_j \in \Lambda
        \cap (4t \Omega_{j,\dwn} \setminus 3t \Omega_{j,\dwn}$ so that
            $|\lambda_{j}^\star - \zeta_j^\star| < t^{-d_\dwn/d_\lef}$. Put
            $\gamma_j = \lambda_j^\up - \zeta_j^\up$, this satisfies all the
            requirements.
        \end{proof}

The following porism is obtained by inspecting the construction used to prove
\cite[Theorem 1.1]{KonyaginSkriganovSobolev} (rather than the statement of the
theorem itself).
\begin{porism}
    \label{por:KSS}
    Let $\F$ be some euclidean space, let $\theta_1, \theta_2 \in \F$ be
    linearly independent, and put $\Theta = \spn_\Z(\theta_1,\theta_2) \subset \F$. There exists $t_0, C
    > 0$
    depending on $\frac{\abs \theta_1}{\abs \theta_2}$ so that for every $t > t_0$ there exists $n, m
    \le \log(t)^7$ such that $\theta := n \theta_1 + m \theta_2$ satisfies
    \begin{equation}
        \dist(2t\abs\theta, \Z) \ge \exp(-C \log \log (t)^4).
    \end{equation}
\end{porism}

We are now ready to prove our main lemma of this section.
\begin{lemma}
    \label{lem:diophantinecandp}
    Let $\Gamma \subset \E$ be a lattice and suppose that $d_\dwn \ge 2$. For every $R_0 > 0$, there exists
    $t_0, c, C > 0$ such that for every $t > t_0$ there is 
    $\theta \in \Gamma$ such that
    \begin{itemize}
        \item $\abs{\theta_\lef} \le R_0$;
        \item $\abs{\theta_\dwn} \le c \log(t)^{8\frac{d_\lef}{d_\dwn} + 7}$;
        \item $\dist(2t \abs{\theta_\dwn},\Z) \ge \exp( -C \log \log(t)^4)$.
    \end{itemize}
\end{lemma}

\begin{proof}
    Applying Lemma \ref{lem:pigeons}, we find $t_0$ such that for every $t >
    t_0$ there are $\theta_1, \theta_2$ such that
    \begin{enumerate}
        \item The projection on $\E_\lef$ is small: $\abs{\theta_{j,\lef}}\le
                R_0 \log(t)^{-8}$;
            \item the projection on $\E_\dwn$ is not too large:
                $\abs{\theta_{j,\dwn}} \lesssim_{\Gamma,R_0} \log(t)^{\frac{8
                d_\lef}{d_\dwn}}$;
            \item the projections on $\E_\dwn$ have comparable norms:
                $\theta_{1,\dwn} \asymp_{\Gamma} \theta_{2,\dwn}$;
            \item the projections on $\E_\dwn$ don't have a small angle between
                them:
                \begin{equation}
                \frac{\theta_{1,\dwn} \cdot
        \theta_{2,\dwn}}{\abs{\theta_{1,\dwn}} \abs{\theta_{2,\dwn}}} \le \frac
        1 2.
                \end{equation}
    \end{enumerate}
    Put $\Theta = \spn_\Z(\theta_{1,\dwn},\theta_{2,\dwn}) \subset \E_\dwn$. By the Porism
    \ref{por:KSS}, there are $n, m \le \log(t)^7$ such that $\lambda := n
    \theta_{1,\dwn} + m \theta_{2,\dwn}$ satisfies
    $\dist(2t|\lambda|,\Z) \ge \exp(- C \log \log(t)^4)$. We put $\theta = n
    \theta_1 + m \theta_2$, whence $\theta_\dwn = \lambda$. We now observe by
    the triangle inequality that $\theta$ is the element whose existence we
    asserted:
    \begin{enumerate}
        \item the projection on $\E_\lef$ remains uniformly bounded:
            $\abs{\theta_\lef} \le 2 R_0 \log(t)^{-1} \le R_0$ as long as $t >
            e^2$;
        \item the projection on $\E_\dwn$ remains not too large:
            $\abs{\theta_{\dwn}} \lesssim \log(t)^{8\frac{d_\lef}{d_\dwn} + 7}$.
    \end{enumerate}
    These properties form our claim. 
\end{proof}

\section{Averaged lower bounds for the discrepancy}\label{sec:average}

In this section, we prove the averaged bounds on the discrepancy. Our first
lemma proves that for every cut and project set, no matter how bad the window or
the search region is, the average of the discrepancy
over translates of the lattice is zero.
\begin{lemma}
    \label{lem:zeroaverage}
    Let $\Omega_\lef \subset \E_\lef$, $\Omega_\dwn \subset \E_\dwn$ and put
    $\Omega = \Omega_\dwn \times \Omega_\lef$. For every lattice $\Gamma$ and $t
    > 0$
    \begin{equation}
        \int_{\E/\Gamma} \Delta(\Omega;\Gamma + \bs;t) \de \bs = 0.
    \end{equation}
\end{lemma}

\begin{proof}
    By definition,
    \begin{equation}
        \label{eq:unfolded}
        \begin{aligned}
            \Delta(\Omega,\Gamma + \bs,t) 
            &=  \sum_{\gamma \in \Gamma} \chi_\Omega(\gamma - \bs) -
            \frac{\vol(\Omega)}{\covol(\Gamma)}. 
        \end{aligned}
    \end{equation}
    Integrating over $\E/\Gamma$ a sum over elements of
    $\Gamma$ is the same as integrating over $\E$ without taking the sum, in
    other words:
    \begin{equation}
        \int_{\E/\Gamma} \Delta(\Omega,\Gamma + \bs,t) \de \bs = -
        \vol(\Omega) + \int_\E \chi_\Omega(-\bs) \de \bs  = 0
    \end{equation}
\end{proof}

\textbf{Lower bounds on the averaged discrepancy, Proof of Theorem
\ref{thm:loweravg}}

When the dimension $d_\dwn = 1$, the claim is simply that there is a constant bound below
for the average of the discrepancy, which is readily seen to hold. We now assume
that $d_\dwn \ge 2$.
This proof is in the same spirit as \cite[Lemma 1]{dahlbergtrubowitz}
and \cite[Section 4]{lagaceparnovski}. Recall that for this theorem
$\Omega_\lef$ is any bounded window with nonzero volume, and $\Omega_\dwn =
\B_\dwn(0,1)$. 
We first observe that $\Delta(\Omega,\Gamma + \bs,t)$
    is a $\Gamma$-periodic function in the parameter $\bs$, it therefore makes
    sense to compute the coefficients at $\gamma^\dagger \in \Gamma^\dagger$ of its Fourier series:
\begin{equation}
    \widetilde{\Delta}(\Omega,\Gamma,t)_{\gamma^\dagger}
    := \int_{\E/\Gamma} \Delta(\Lambda(\Gamma;\bs;\Omega_\lef))
    \be(\gamma^\dagger \cdot \bs) \de \bs.
\end{equation}
    Lemma \ref{lem:zeroaverage} tells
    us that the zero'th Fourier coefficient vanishes, for other $\gamma^\dagger \in
    \Gamma^\dagger \setminus \set 0$ we obtain
    in the same way as
    in \eqref{eq:unfolded} 
    \begin{equation}
        \begin{aligned}
            \widetilde{\Delta}(\Omega,\Gamma,t)_{\gamma^\dagger} &=
        \int_{\E/\Gamma} \left(- \frac{\vol(\Omega)}{\covol(\Gamma)} + \sum_{\gamma
        \in \Gamma} \chi_{t\Omega_\dwn \times \Omega_\lef}(\gamma - \bs)\right)\be(\gamma^\dagger
        \cdot \bs) \de \bs \\
        &=                 \int_\E \chi_{t\Omega_\dwn \times \Omega_\lef}(\bs) \be(- \gamma^\dagger \cdot \bs)
    \de \bs \\
    &=  [\CF \chi_{t\Omega_\dwn \times \Omega_\lef}](-\gamma^\dagger) \\
    &=  [\CF_\dwn \chi_{t\Omega_\dwn}](- \gamma^\dagger_\dwn) [\CF_\lef
    \chi_{\Omega_\lef}](-\gamma^\dagger_\lef) \\
    &= t^d [\CF_\dwn \chi_{\Omega_\dwn}](-t\gamma^\dagger_\dwn) [\CF_\lef
    \chi_{\Omega_\lef}](- \gamma^\dagger_\lef).
        \end{aligned}
    \end{equation}
    It follows from the triangle inequality that,
    \begin{equation}
        \label{eq:comparisonL1linfty}
        \norm{\Delta}_{\RL^1(\E/\Gamma)} \ge \norm{\tilde
        \Delta}_{\ell^\infty(\Gamma^\dagger)},
    \end{equation}
    so that finding large Fourier coefficients implies lower bounds on the
    $\RL^1$ norm of the discrepancy. 
    Let us first examine the term $\CF_\lef \chi_{\Omega_\lef}$. We assumed
    $\Omega_\lef$ is bounded, for definiteness suppose that
    $\Omega_\lef$ is contained in a ball of radius $R$ in $\E_\lef$. Then, 
    \begin{equation}
        \label{eq:xideriv}
        \abs{\nabla_{\xi_\lef} \CF_\lef \chi_{\Omega_\lef}(\xi_\lef)} \le \int_{\E_\lef}
        \abs{x_\lef} \chi_{\Omega_\lef}(x_\lef) \de x_\lef \le \Vol(\Omega_\lef)
        R
    \end{equation}
    so that since $[\CF_\lef \chi_{\Omega_\lef}](0) = \Vol(\Omega_\lef)$,
    Taylor's theorem implies that as long as $|\gamma^\dagger_\lef| \le (2R)^{-1}$,
    \begin{equation}
        \label{eq:crudecomparison}
    \abs{[\CF_{\lef} \chi_{\Omega_\lef}](- \gamma_\lef^\dagger)} \ge \frac{\Vol(\Omega_\lef)}{2}.
    \end{equation}
    On the other hand, the Fourier transform of the indicator of a unit ball can
    be explicitly computed in terms of Bessel functions, for which there are
    precise asymptotic descriptions. Following \cite[Formula
    8.451.1]{gradshteynryzhik}, there is $z_0$ such that as long as $t
|\gamma_\dwn^\dagger| \ge z_0$ 
    \begin{equation}
        \label{eq:fourierunitball}
        \begin{aligned}
            t^{d_\dwn} [\CF_\dwn \chi_{\Omega_\dwn}](-t \gamma^\dagger_\dwn) &= \frac{
        t^{d_\dwn/2}}{ |\gamma^\dagger_\dwn|^{d_\dwn/2}}
            J_{d_\dwn/2}\left( 2 \pi t|\gamma^\dagger_\dwn|\right) \\
            &= \frac{
            t^\frac{d_\dwn-1}{2}}{2\pi |\gamma^\dagger_\dwn|^{\frac{d_\dwn+1}{2}}} \sin\left(2
                \pi t |\gamma^\dagger_\dwn| + \frac{1-d_\dwn}{4} \pi \right) +
                \bigo{t^{\frac{d_\dwn-3}{2}}
                |\gamma^\dagger_\dwn|^{\frac{-d_\dwn-3}{2}}}.
        \end{aligned}
    \end{equation}
    \textbf{Case 1, $d_\dwn \not \equiv 1 \mod 4$:} In this situation, we use that 
    \begin{equation}
        \label{eq:smallsin}
       0 < \inf_{x \in \R} \max\set{
            \abs{\sin\left(x + \frac{1 -
        d_\dwn}{4}\pi\right)},
            \abs{\sin\left(2x + \frac{1 -
        d_\dwn}{4}\pi\right)}} 
    \end{equation}
    so that for any fixed $\gamma^\dagger \in \Gamma^\dagger$ with $|\gamma_\lef^\dagger| \le (4
    R)^{-1}$,
    \begin{equation}
        \max\set{
            \abs{t^{d_\dwn}\CF_\dwn \chi_{\Omega_\dwn}(-t \gamma_\dwn^\dagger)},
        \abs{t^{d_\dwn}\CF_\dwn \chi_{\Omega_\dwn}(-2t \gamma_\dwn^\dagger)}} \gtrsim
        t^{\frac{d_\dwn-1}{2}},
    \end{equation}
    and inserting this along with the estimate \eqref{eq:crudecomparison} in
    \eqref{eq:comparisonL1linfty} is our claim. 

    \textbf{Case 2, $d_\dwn \equiv 1 \mod 4$:}
    Here the
    situation is more delicate as \eqref{eq:smallsin} doesn't hold anymore.
    In this situation we read from \eqref{eq:fourierunitball} that as long as
    $\dist(2 t |\gamma^\dagger_\dwn|, \Z) \ge t^{-1}
    |\gamma_\dwn^\dagger|^{-1}$,
    \begin{equation}
        \label{eq:estimate1mod4}
        \begin{aligned}
            \abs{t^{d_\dwn}[\CF_\dwn \chi_{\Omega_\dwn}](-t \gamma^\dagger_\dwn) } &=
        \abs{\frac{t^{\frac{d_\dwn-1}{2}}}{2\pi
            |\gamma_\dwn^\dagger|^{\frac{d_\dwn+1}{2}}}
        \sin\left(2 \pi t |\gamma_\dwn^\dagger|\right)} +
        \bigo{t^{\frac{d_\dwn-3}{2}}|\gamma_\dwn^\dagger|^{\frac{-d_\dwn-3}{2}}} \\
    &\gtrsim \frac{t^{\frac{d_\dwn-1}{2}}}{|\gamma_\dwn^\dagger|^{\frac{d_\dwn+1}{2}}}
    \dist(2 t |\gamma_\dwn^\dagger|,\Z).
    \end{aligned}
    \end{equation}
    Following Lemma \ref{lem:diophantinecandp} there exists $\gamma^\dagger \in
    \Gamma^\dagger$ such that 
    \begin{equation}
        |\gamma^\dagger_\lef| < (2R)^{-1}, \qquad
        |\gamma^\dagger_\dwn| \lesssim \log(t)^{8\frac{d_\lef}{d_\dwn} +7}, \quad  \text{and}
        \quad \dist(2t
    |\gamma_\dwn^\dagger|,\Z) \ge \exp(-C \log \log(t)^4).
\end{equation}
Inserting into \eqref{eq:estimate1mod4}, along with the observation that for
every $s \ge 0$
\begin{equation}
    \log(t)^{-s} \exp (-C \log \log(t)^4) \gtrsim_{\eps,s} \exp(- (C + \eps) \log \log(t)^4)
\end{equation}
yields the existence of a constant $A > 0$ such that 
\begin{equation}
    \abs{t^{d_\dwn}[\CF_\dwn \chi_{\Omega_\dwn}](-t \gamma^\dagger_\dwn) } \gtrsim
        t^{\frac{d_\dwn-1}{2}} \exp(- A \log\log(t)^4).
\end{equation}

\begin{remark}
    \label{rem:boundedtoostrong}
    It follows from \eqref{eq:xideriv} that we would be able to weaken the
    assumptions on the window $\Omega_\lef$ to
    \begin{equation}
        \int_{\Omega_\lef} \abs{\bx_\lef} \de \bx_\lef < \infty
    \end{equation}
    and obtain the same lower bounds on the average of the discrepancy. While
    this allows for some unbounded windows, we decided to keep the statement for
    bounded windows to keep in line with the standard assumptions of the field.
\end{remark}
\newpage

\appendix


\section{\texorpdfstring{$L^2$}{L2}-bounds for the discrepancy of Liouville
    cut and project sets \texorpdfstring{\begin{center} $\phantom{p}$ \\[-15pt] by Michael
Bj\"orklund and Tobias Hartnick\end{center}}{}}

\pagestyle{mystyleappendix}

\vspace{-15pt}

\subsection{Statement of the main result}\label{AppMain}

As in the body of the text we fix an orthogonal decomposition 
\[
\R^d \cong \E := \E_\dwn \oplus
\E_\lef \cong \R^{d_\dwn} \oplus \R^{d_\lef}
\]
with associated projections
$\pi_\dwn$ and $\pi_\lef$ and a lattice $\Gamma \subset
\E$ such that $\pi_\lef(\Gamma)$ is dense in $\E_\lef$ and
$\pi_\dwn\big|_{\Gamma}$ is injective. To simplify notation we will assume
without loss of generality that $\E = \R^d$ and that $\E^{\dwn} = \R^{d_\dwn}
\times \{0\}$ and $\E^{\lef} = \{0\} \times \R^{d_\lef}$. We write elements of
$\E$ as $\bs = (\bs_\dwn, \bs_\lef)$, where $\bs_\dwn \in \R^{d_\dwn}$ and
$\bs_\lef \in \R^{d_\lef}$. Given $\bs \in \E$ and a bounded \emph{window}
$\Omega_\lef \subset \E_\lef$ we denote by
\begin{equation}
 \Lambda(\Gamma + \bs;\Omega_\lef) := \pi_\dwn((\Gamma + \bs) \cap
    \pi_\lef^{-1}(\Omega_\lef))
\end{equation}
the associated cut and project set. If $\Lambda$ is such a cut and project set and $\Omega_\dwn \subset \E_\dwn$ is a bounded Borel subset, then we define the \emph{discrepancy} of $\Lambda$ with respect to $\Omega_\dwn$ as
\begin{equation}
    \label{eq:appdisc}
 \Delta(\Lambda; \Omega_\dwn) = \#(\Lambda \cap \Omega_\dwn) -  \frac{\vol_\dwn(\Omega_\dwn)\vol_\lef(\Omega_\lef)}{\covol(\Gamma)}.
 \end{equation}
 In the study of the discrepancy of cut and project sets one can vary the
underlying lattice $\Gamma$, underlying window $\Omega_\lef$ and the bounded set
$\Omega_\dwn$. We will always assume that $\Omega_\dwn$ is a fixed set of bounded perimeter and study its dilates
 $t\Omega_\dwn$. In this specific case we have $\vol_\dwn(t\Omega_\dwn) =
 t^{d_\dwn}\vol_\dwn(\Omega_\dwn)$, and hence formula \eqref{eq:appdisc} reduces to the one from
     Theorem \ref{thm:main} in the introduction. Moreover we have $\Delta(\Lambda; \Omega_\dwn) =
 o(t^{d_\dwn})$ which means that the number of points in dilates of
 $\Omega_\dwn$ is
 asymptotically proportional to the volume of $\Omega_\dwn$ with an error term given
 by the discrepancy. 
If we fix the window $\Omega_\lef$ and only vary the lattice by translations,
then we can consider the discrepancy as a function
\[
\Delta_{\Omega_\lef}: \Gamma \backslash \E \times [0, \infty) \to \R, \quad
(\Gamma + \bs, t) \mapsto \Delta_{\Omega_\lef}(\Gamma + \bs, t) := \Delta(
\Lambda(\Gamma + \bs;\Omega_\lef),  t\Omega_\dwn).
\]
In the context of the bi-Lipschitz problem for model sets one is mostly interested in estimates of the form
\[
\phi^-_{\Omega_\lef}(t) \leq \sup_{\Gamma + \bs \in  \Gamma \backslash \E}
|\Delta_{\Omega_\lef}(\Gamma + \bs, t)| \leq \phi^+_{\Omega_\lef}(t),
\]
which are \emph{uniform} in the torus variable. However, there is also some
interest in \emph{$\RL^p$-estimates} of the form
\[
\phi^{p,-}_{\Omega_\lef}(t)  \leq \int_{ \Gamma \backslash \E } |
\Delta_{\Omega_\lef}(\Gamma + \bs, t)|^p \; \mathrm{d}\bs  \leq \phi^{p,+}_{\Omega_\lef}(t).
\]
This appendix is specifically concerned with lower $\RL^2$-bounds for the
discrepancy for specific choices of $\Gamma$ and $t$ (and $\Omega_\lef$ given by
 Euclidean balls of some ``generic'' radius or cylinders of some ``generic''
 sidelength), i.e.\ we will provide lower bounds for the \emph{number variance}
\begin{equation}
\mathrm{NV}_t(\Gamma, \Omega_\lef) := \int_{ \Gamma \backslash \E } | \Delta_{\Omega_\lef}(\Gamma + \bs, t)|^2 \; \mathrm{d}\bs.
\end{equation}
As the term ``number variance'' suggests, this $\RL^2$-norm can be interpreted
as the variance of a certain random variable and hence can be estimated using
probabilistic tools. Before we discuss this probabilistic interpretation, we
state our main result. For this we need the following notion. We denote by $\|x\|$ the nearest distance from a real number $x$ to an integer and
say that an increasing function $\psi : (0,\infty) \to (0,\infty)$ is \emph{increasing
    slowly at infinity} if $\psi(r) \lesssim_\eps r^\eps$ for every
    $\eps > 0$ as $r \to \infty$. Typical examples of such functions are given by $\psi(r) = \log(1+r)^\beta$ for some fixed $\beta > 0$ or $\psi(r) = \log\log(e + r)$.
\begin{defi} Let $a \in \R$ and let $\psi : (0,\infty) \to (0,\infty)$ be a function which is increasing
    slowly at infinity.
\begin{itemize}
\item[(i)] $a$ is a \emph{Liouville number} if there is a sequence of integers $(q_n)$ so that
\[
 q_n \to \infty \qand \|q_n a \|^{-1/n} \ge q_n.
\] 
\item[(ii)] $a$ is a  \emph{$\psi$-Liouville number} if  there is a sequence of integers $\set{q_n}$ so that  
    \begin{equation}
        q_n \to \infty \qand \psi(\|q_n a \|^{-1}) \ge q_n.
    \end{equation}
\end{itemize}
\end{defi}
\begin{remark} Every $\psi$-Liouville number is a Liouville number, and conversely every Liouville number is $\psi$-Liouville for \emph{some} $\psi$. Indeed, we can just set 
 $\psi(\|q_n a\|^{-1}) := \|q_na\|^{-1/n}$, and
interpolate linearly. There exist $\psi$-Liouville numbers for any $\psi$ which increases slowly at infinity.
\end{remark}
There is a similar notion for lattices:
\begin{defi}
    Let $\psi: (0,\infty) \to (0,\infty)$ be slowly increasing at infinity. A
    lattice $\Gamma \subset \E$
        is said to be \emph{$\psi$-Liouvillean} (with respect to the splitting $\E =
        \E_\dwn \oplus \E_\lef$) if it is irrational with respect
        to $\E_\dwn \oplus \E_\lef$ and there is a sequence $\set{\gamma^{(n)}} =
        \set{(\gamma_\dwn^{(n)},\gamma_\lef^{(n)})} \subset \Gamma
            \setminus \set{(0,0)}$ such that
\begin{equation}\label{LiouvilleParameters}
\lim_{n \to \infty} \gamma_\dwn^{(n)} = 0 \qand \lim_{n \to \infty} |\gamma_\lef^{(n)}| = \infty
\qand |\gamma_\lef^{(n)}| \leq \psi(|\gamma_\dwn^{(n)}|^{-1}) \quad \textrm{for all $n$}.
\end{equation}
It is called a \emph{Liouvillean} if it is $\psi$-Liouvillean for some $\psi$.
\end{defi}
\begin{example}\label{ex:Liouvillean}  We can construct $\psi$-Liouvillean lattices for any slowly increasing $\psi$: Let $e_\dwn \in \E_\dwn, e_\lef \in \E_\lef$ be
    unit vectors and further decompose the splitting $\E_\dwn \oplus \E_\lef$ as
    \begin{equation}
       \E = \spn(e_\dwn) \oplus \F_\dwn \oplus \spn(e_\lef) \oplus \F_\lef,
    \end{equation}
    where $\F_\bullet = \E_\bullet \ominus \spn(e_\bullet)$. Putting $\F =
    \F_\dwn \oplus \F_\lef$, let $\tilde \Gamma
    \subset \F$ be a lattice which is irrational with respect to $\F_\dwn \oplus
    \F_\lef$. If $\frac 1 2 < a < \frac 3 2$ is a $\psi$-Liouville number and
    $\frac 3 4 < b < \frac 5 4$ is irrational, then the lattice defined by
    \begin{equation}
        \Gamma = \spn_\Z(a e_\dwn + e_\lef, e_\dwn + b e_\lef) \oplus \tilde\Gamma \subset \E
    \end{equation}
    is a $\psi$-Liouvillean lattice with respect to the decomposition $\E_\dwn
    \oplus \E_\lef$. Indeed, it is easy to see that it is a lattice irrational
    with respect to $\E_\dwn \oplus \E_\lef$. To see that it 
    it is $\psi$-Liouvillean, consider a sequence $q_n\to \infty$ so that
    $\psi(\|aq_n\|^{-1}) \ge q_n$, and put $m_n$ to be the nearest integer to
    $aq_n$. Then, 
    $$\gamma^{(n)} = q_n (a e_\dwn + e_\lef) - m_n (e_\dwn +
    e_\lef) \in
    \Gamma$$
    is so that
    \begin{equation}
        |\gamma^{(n)}_\dwn| = \|q_n a\| \qand |\gamma^{(n)}_\lef| = |q_n - b m_n| <
        |q_n|,
    \end{equation}
and this sequence makes $\Gamma$ a $\psi$-Liouvillean lattice.
\end{example}
\begin{theorem}\label{ThmLiouville}
Let   $f:(0, \infty) \to (0, \infty)$ be a function diverging at infinity, let $\psi:(0, \infty) \to (0, \infty)$ be a slowly increasing function, and let $\Omega_\dwn \subset \E_\dwn$ be of
    bounded perimeter. Then there exist a lattice $\Gamma$, a window $\Omega_\lef$, an increasing sequence $(t_n)$ of positive real numbers such that
    \begin{equation}\label{AppGrowth}
    \varlimsup_{n \to \infty} \, \frac{\mathrm{NV}_{t_n}(\Gamma, \Omega_\lef)}{t_n^{2d_\dwn}\psi(t_n)^{- d_\lef - 1} f(t_n)^{-1}} = \infty.
    \end{equation}
 In fact, \eqref{AppGrowth} holds whenever the dual lattice $\Gamma^\dagger$ of $\Gamma$ is $\psi$-Liouvillean and $\Omega_\lef = \B_\lef(0,r)$ for $r$ in some Lebesgue conull subset $A \subset (0, \infty)$ depending on $\Gamma^\dagger$.
\end{theorem}
For the dual lattices of the specific lattices constructed in Example \ref{ex:Liouvillean} and windows adapted to these examples one can also get a lower bound, which matches the upper bounds obtained in the main body of this article:
\begin{theorem}\label{ThmLiouvillebis}
Let $f : (0,\infty)  \to (0,\infty)$ be a function diverging at infinity, let $\psi$ be a slowly growing function and let $\Gamma$ be the dual of the $\psi$-Liouvillean lattice from
    Example \ref{ex:Liouvillean}. Moreover, let $\Omega_\dwn \subset \E_\dwn$ be of bounded perimeter and let  $\Omega_\lef = [-r,r]
    \times \Upsilon_\lef \subset \E_\lef =
    \spn(e_\lef) \oplus \F_\lef$ for some $\Upsilon_\lef$ of bounded perimeter. Then
    there is a sequence $\set{t_n}$ of positive real
    numbers such that $t_n \to \infty$ and a Lebesgue conull subset $A \subset
    (0,\infty)$ such that for all $r \in A$ and all functions $f : (0,\infty)
    \to (0,\infty)$ diverging at infinity
\[
\varlimsup_{n \to \infty} \, \frac{\mathrm{NV}_{t_n}(\Gamma,
\Omega_\lef)}{t_n^{2d_\dwn}\psi(t_n)^{- 2} f(t_n)^{-1}} = \infty.
\]
\end{theorem}

\subsection{The diffraction formula for the number variance}
It will be convenient for us to use the language of point processes; see e.g.\ \cite{Kallenberg} for background. Given $n \in \mathbb N$, we denote by $M(\R^n)$ the space of Radon measures on $\R^n$ and by $\mathcal{LF}(\R^n)$ the space of locally finite (i.e.\ closed and discrete) subsets of $\R^n$. We consider $\mathcal{LF}(\R^n)$ as a subset of $M(\R^n)$ by identifying each set $\Lambda$ with its Dirac comb $\delta_\Lambda := \sum_{x \in \Lambda} \delta_x$. Every bounded Borel function $f: \R^n \to \mathbb C$ with bounded support defines a \emph{linear statistic} 
\[\mathcal P f: M(G) \to \mathbb C, \quad p \mapsto p(f),\]
and we equip $\mathcal{LF}(\R^n) \subset M(\R^n)$ with the smallest $\sigma$-algebra $\mathcal B$ for which all of these linear statistics are measurable. Now assume that $(\Omega, \mathcal F, \mathbb{P})$ is some auxiliary probability space on which $\R^n$ acts measurably, preserving $\mathbb P$; then an equivariant measurable map
\[
\Lambda: \Omega \to \mathcal{LF}(\R^n), \quad \omega \mapsto \Lambda_\omega.
\]
is called a \emph{stationary simple point process} with \emph{distribution} $\mu = \Lambda_*\mathbb{P}$. We will only consider processes which are \emph{locally square-integrable} in the sense that the real-value random variable
\[
\Lambda \cap B: \Omega \to \R, \quad z \mapsto \# (\Lambda_z \cap B)
\]
is square-integrable for every bounded Borel set $B \subset \R^n$; note that this is automatically the case if $\Lambda_\omega$ is almost surely $r$-uniformly discrete for some fixed $r>0$. 
\begin{defi}
If $\Lambda: \Omega \to \mathcal{LF}(\R^n)$ is locally square-integrable, then the random variable
\[
\mathrm{disc}_B(\Lambda): \Omega \to \R, \quad \mathrm{disc}_B(\Lambda)(z) := \#(\Lambda_z \cap B) - E[\Lambda \cap B].
\]
is called the \emph{discrepancy} of $\Lambda$ on $B$, and the variance
\[
N_B(\Lambda) :=  \mathrm{Var(\Lambda \cap B)} =  \int_\Omega |\mathrm{disc}_B(\Lambda)(\omega)|^2 \, \mathrm{d}\mathbb P(\omega)
\]
If $\Lambda: \Omega \to \mathcal{LF}(\R^n)$ is locally square-integrable, then the random variable
\[
\mathrm{disc}_B(\Lambda): \Omega \to \R, \quad \mathrm{disc}_B(\Lambda)(z) := \#(\Lambda_z \cap B) - E[\Lambda \cap B].
\]
is called the \emph{discrepancy} of $\Lambda$ on $B$, and the variance
\[
N_B(\Lambda) :=  \mathrm{Var(\Lambda \cap B)} =  \int_\Omega |\mathrm{disc}_B(\Lambda)(\omega)|^2 \, \mathrm{d}\mathbb P(\omega)
\]
is called the \emph{number variance} of $\Lambda$ at $B$.
\end{defi}
Note that the number variance is finite, since $\Lambda \cap B$ is square-integrable. 
\begin{remark}[Diffraction formula]
According to \cite[Prop.\ 2.3]{BjorklundHartnick} there exists a positive-definite signed Radon measure $\eta_\Lambda$, the \emph{(reduced) autocorrelation measure} of $\Lambda$, such that
\begin{equation}\label{NBAuto}
N_B(\Lambda) = \eta_\Lambda(\chi_B \ast \chi_{B}^*),
\end{equation}
where $\chi_B$ denotes the characteristic function of $B$. The Fourier transform
$\CF\eta_\Lambda$ is a positive Radon measure on $\R^n$, called the
\emph{(reduced) diffraction} of $\Lambda$, and it follows from \eqref{NBAuto}
that for sufficiently regular $B$ we have
\begin{equation}\label{NBDiff}
    N_B(\Lambda) = [\CF\eta_\Lambda](|\CF \chi_B|^2).
\end{equation}
This is called the \emph{diffraction formula} for the number variance.
\end{remark}
\begin{example} We return to the general setting of Subsection \ref{AppMain}. If we fix the lattice $\Gamma \subset \E$ and the window $\Omega_\lef$ then the map
\[
\Lambda(-; \Omega_\lef): \Gamma \backslash \E \to \mathcal{LF}(\E_\dwn), \quad \Gamma + \bs \mapsto \Lambda(\Gamma + \bs; \Omega_\lef)
\]
is a stationary simple point process with respect to the unique $\E$-invariant probability measure on $\Gamma \backslash \E$, called the \emph{cut and project process} associated with $\Gamma$ and $\Omega_\lef$. For a bounded Borel set $\Omega_\dwn \subset \E_\dwn$ we have 
\[
 E[\Lambda(-; \Omega_\lef) \cap B] =  \frac{\vol_\dwn(\Omega_\dwn)\vol_\lef(\Omega_\lef)}{\covol(\Gamma)},
\]
and hence the discrepancy of $\Lambda(-;\Omega_\lef)$ is given by
\[
\mathrm{disc}_{\Omega_\dwn}(\Lambda(-; \Omega_\lef))(\Gamma + \bs) =  \Delta_{\Omega_\lef}(\Gamma + \bs; \Omega_\dwn).
\]
On the other hand, it follows from classical results of Meyer \cite{Meyer70} that the reduced diffraction of $\Lambda(-;\Omega_\lef)$ is given by \cite[Thm.\ 2.9]{BjorklundHartnick}
\[
\CF\eta = \frac{1}{\covol(\Gamma)^2}  \, \sum_{\xi = (\xi_\dwn, \xi_\lef) \in
\Gamma^\dagger \setminus\{(0,0)\}} |[\CF \chi_{\Omega_\lef}](\xi_\lef)|^2 \cdot \delta_{\xi_\dwn}.
\]
The diffraction formula \eqref{NBDiff} for the number variance of
$\Lambda(-;\Omega_\lef)$ may thus be stated as follows; here we use that since
$\chi_{t\Omega_\dwn}(x_\dwn) = \chi_{\Omega_\dwn}(x_\dwn/t)$ we have
$[\CF \chi_{t\Omega_\dwn}](\xi_\dwn) =
t^{d_\dwn}[\CF\chi_{\Omega_\dwn}](t\xi_\dwn)$.
\end{example}
\begin{prop}\label{AppMainFormula} If $\Omega_\lef$ has sufficient Fourier decay, then for all $t>0$ we have
\[
\mathrm{NV}_t(\Gamma, \Omega_\lef) = \frac{t^{2 d_\dwn}}{\covol(\Gamma)^2}  \,
\sum_{\xi = (\xi_\dwn, \xi_\lef) \in \Gamma^\dagger \setminus\{(0,0)\}}
|[\CF\chi_{\Omega_\dwn}](t \xi_\dwn)|^2 |[\CF\chi_{\Omega_\lef}](\xi_\lef)|^2.\qed
\]
\end{prop}

\subsection{Proof of Theorem \ref{ThmLiouville}}
We now return to the setting of Theorem \ref{ThmLiouville}; thus $\Gamma^\dagger$ is
assumed to be $\psi$-Liouvillean with parameters $\psi, \xi_\dwn^{(n)},
\xi_\lef^{(n)}$ as in \eqref{LiouvilleParameters} and $\Omega_{\lef} =
\B_\lef(0,r)$ for some $r>0$. Let $f : (0,\infty) \to (0,\infty)$ be a function
such that $\lim_{t \to \infty} f(t) = \infty$. If we abbreviate $\B_\lef := \B_\lef(0,1)$,
then by Proposition \ref{AppMainFormula} we have
\begin{equation}\label{DefPsi}
    \Psi_{f}(t,r) := \frac{\mathrm{NV}_{t}(\Gamma,
    \B_\lef(0,r))}{t^{2d_\dwn}\psi(t)^{- d_\lef - 1}f(t)^{-1}} =
    \frac{\psi(t)^{d_\lef + 1}f(t) r^{2d_\lef}}{\covol(\Gamma)^2}\,  \sum_{\xi =
    \in \Gamma^\dagger \setminus\{(0,0)\}} |[\CF\chi_{\Omega_\dwn}](t
    \xi_\dwn)|^2 |[\CF{\chi}_{ \B_\lef}](r\xi_\lef)|^2.
\end{equation}
We have to prove that there exists a sequence $t_n$ such that for all $r$ in a Lebesgue conull subset $E \subset (0,\infty)$ we have
\begin{equation}\label{AppendixToShow}
\varlimsup_{n \to \infty}\Psi_f(t_n, r) = \infty.
\end{equation}
Following \eqref{eq:xideriv},
\begin{equation}
    \label{eq:boundxidown}
    \inf\set{
|\widehat{\chi}_{\B_\dwn}(\xi_\dwn)| :
|\xi_\dwn| \leq (2\diam(\Omega_\dwn))^{-1}}
\geq \frac 1 2 \vol_\dwn(\B_\dwn).
\end{equation}
If we set
\[
\mathcal C_t := \{\xi = (\xi_\dwn, \xi_\lef) \in \Gamma^\dagger
\setminus\{(0,0)\} : |\xi_\dwn| \leq 1/t\},
\]
we thus obtain for any $t \ge (2 \diam(\Omega_\dwn))^{-1}$
\begin{eqnarray*}
    \Psi_f(t,r) &\geq&  \frac{\psi(t)^{d_\lef + 1} f(t) r^{2d_\lef}}{\covol(\Gamma)^2}\,
    \sum_{\xi \in \mathcal C_t} |[\CF{\chi}_{\Omega_\dwn}](t \xi_\dwn)|^2
    |[\CF{\chi}_{\B_\lef}](r\xi_\lef)|^2 \\
&\geq&  \frac{\psi(t)^{d_\lef + 1 } f(t) r^{2d_\lef}  \vol_\dwn(\Omega_\dwn)^2}{4\covol(\Gamma)^2}\,
\sum_{\xi  \in \mathcal C_t} |[\CF{\chi}_{ \B_\lef}](r\xi_\lef)|^2\\
&\geq&  \frac{\psi(t)^{d_\lef + 1} f(t) r^{2d_\lef}
\vol_\dwn(\Omega_\dwn)^2}{4\covol(\Gamma)^2}\, \sup_{\xi \in \mathcal C_t}
|[\CF{\chi}_{ \B_\lef}](r\xi_\lef)|^2.
\end{eqnarray*}
\begin{lemma}
    If we define $t_n := \|\xi_\dwn^{(n)}\|^{-1}$, then there exists a constant
    $C > 0$ depending only on $d_\dwn$ and the covolume of $\Gamma$ such that
\begin{equation}\label{MainInequalityAppendix}
    \Psi_f(t_n, r) \geq C \cdot |\xi_\lef^{(n)}|f(\psi^{-1}(|\xi_\lef^{(n)}|))  \cdot |J_{d_\lef/2}(2\pi r |\xi^{(n)}_\lef|)|^2.
\end{equation}
\end{lemma}
\begin{proof} Since $\xi_\dwn^{(n)} \to 0$ we have $t_n \to \infty$, and for all
    $n \in \mathbb N$ sufficiently large we have $|\xi_\dwn^{(n)}| \leq 1/t_n$,
    and hence $\xi_n \in \mathcal C_{t_n}$. We deduce that, for all $n \in
    \mathbb N$ large enough we have
\begin{equation}
    \sup_{\xi \in \mathcal C_{t_n}}  |[\CF{\chi}_{ \B_\lef}](r\xi_\lef)|^2 \geq
    |[\CF \chi_{ \B_\lef}](r\xi^{(n)}_\lef)|^2.
\end{equation}
Moreover, by \eqref{LiouvilleParameters} we have $t_n \geq
\psi^{-1}(|\xi_\lef^{(n)}|)$. Plugging both into the inequality above, we deduce that 
\begin{equation}
    \begin{aligned}
        \Psi_f(t_n,r) &\geq  \frac{\psi(t_n)^{d_\lef + 1}f(t_n)r^{2d_\lef}
    \vol_\dwn(\Omega_\dwn)^2}{4\covol(\Gamma)^2}\, |[\CF{\chi}_{
    \B_\lef}](r\xi^{(n)}_\lef)|^2 \\
    &\geq
    \frac{\|\xi_\lef^{(n)}\|^{d_\lef + 1}
        f(\psi^{-1}(|\xi_\lef^{(n)}|)) r^{2d_\lef}
    \vol_\dwn(\B_\dwn)^2}{4\covol(\Gamma)^2}\, |[\CF{\chi}_{
    \B_\lef}](r\xi^{(n)}_\lef)|^2.
\end{aligned}
\end{equation}
We recall, following \cite[Formula 8.145.1]{gradshteynryzhik}, that the Fourier
transform of the indicator of a unit ball can be expressed in terms of Bessel functions as
\[
[\CF{\chi}_{B_\lef}](\xi_\lef) = \frac{J_{d_\lef/2}(2\pi |\xi_\lef|)}{|\xi_\lef|^{d_\lef/2}},
\]
hence we obtain the desired inequality with $C = \frac{\vol_\dwn(\Omega_\dwn)^2}{4\covol(\Gamma)^2}$.
\end{proof}
Let us put $\tau_n := 2\pi |\xi^{(n)}_\lef|$ and $\tilde f = f \circ
\psi^{-1}$. By
\eqref{MainInequalityAppendix} we thus have
\[
\Psi_f(t_n, r)  \gtrsim  [\tau_n^{1/2} \tilde  J_{d_\lef/2}(r\tau_n)]^2
\tilde    f(\tau_n).
\]
To deduce \eqref{AppendixToShow} and thereby prove Theorem
\ref{ThmLiouville} it thus suffices to establish the following lemma.
\begin{lemma} 
    Let $(\tau_n)$ be a sequence of positive real numbers such that $\tau_n \to \infty$. 
Then there is a Lebesgue conull subset $E \subset (0,\infty)$ such that
for all $r \in E$, and $d_\lef \in \N$ 
\[
    \varlimsup_{n \to \infty} \tau_n^{1/2}  \cdot |J_{d_\lef/2}(r \tau_n)| \cdot
    \tilde f(\tau_n)= \infty.
\]
\end{lemma}
\begin{proof} Since $\tau_n \to \infty$ we may without loss of generality assume
    that $\tau_n > 1$ for all $n$. By \cite[Subsection 7.21]{Wat}, there exist
    two real constants $\phi$ and $\kappa$ depending on $d_\lef$ such that
\[
    \left| J_{d_\lef/2}(\tau) - \left(\frac{2}{\pi \tau}\right)^{1/2} \cos(\tau-\phi) \right| 
\leq \frac{\kappa}{\tau^{3/2}}, \quad \textrm{for all $\tau \geq 1$.}
\]
In particular, for all $r > 0$ and $n$, 
\begin{align*}
    |J_{d_\lef/2}(r \tau_n)| 
    &= \left|J_{d_\lef/2}(r \tau_n) - \left(\frac{2}{\pi r \tau_n}\right)^{1/2}
    \cos(r \tau_n-\phi)  + \left(\frac{2}{\pi r \tau_n}\right)^{1/2} \cos(r
    \tau_n-\phi) \right| \\[0.2cm]
&\geq 
\left(\frac{2}{\pi r \tau_n}\right)^{1/2} |\cos(r \tau_n-\phi)| - \frac{\kappa}{(r \tau_n)^{3/2}},
\end{align*}
and thus,
\[
    \tau_n^{\frac{1}{2}} \tilde f(\tau_n) \cdot |J_{d_\lef/2}(r \tau_n)| \geq
    \tilde f(\tau_n) \cdot \left(\frac{2}{\pi r }\right)^{1/2} |\cos(r
    \tau_n-\phi)| - \tilde f(\tau_n)\tau_n^{-1}\frac{\kappa}{r^{3/2}}.
\]
Since $\tau_n \to \infty$,  \cite[Chapter 1, Theorem 4.1]{Kui} tells us that we
can find a Lebesgue conull subset $A \subset (0,\infty)$ such that the sequence
$(r\tau_n)$ is equidistributed modulo $2\pi$ for all $r \in A$. Hence, for all
$r \in A$, we can find a further sub-sequence $(\tau_{n_k})$ such that $r
\tau_{n_k} - \phi \to 0$ (modulo $2\pi$) as $k \to \infty$. In
particular, 
\[
|\cos(r \tau_{n_k}-\phi)| \geq \frac{1}{2} \quad \textrm{for all large enough $k$},
\]
and thus
\[
    \tau_{n_k}^{\frac{1}{2}} \tilde f(\tau_{n_k}) \cdot |J_{d_\lef/2}(r \tau_{n_k})|
    \geq \frac{1}{2} \cdot \tilde f(\tau_{n_k}) \cdot \left(\frac{2}{\pi r
    }\right)^{1/2}  - \tilde f(\tau_{n_k})\tau_{n_k}^{-1}\frac{\kappa}{r^{3/2}}.
\]
The right-hand side clearly tends to infinity as $k \to \infty$, which finishes the proof.
\end{proof}

\color{black}

\subsection{Proof of Theorem \ref{ThmLiouvillebis}}
We now turn our eye to Theorem \ref{ThmLiouvillebis}; thus
$\Gamma^\dagger$ is the lattice from Example \ref{ex:Liouvillean}. In particular
there is a $\psi$-Liouville number $0 < a < 2$ and unit length vectors $e_\lef \in \E_\lef$ and $e_\dwn \in \E_\dwn$ such that
$\spn(a e_\dwn + e_\lef, e_\dwn + e_\lef) \subset \Gamma^\dagger$, and the rest
of the lattice is orthogonal to this span. We fix once and for all the sequence
$\xi^{(n)} = q_n (a e_\dwn + e_\lef) - m_n(e_\dwn + e_\lef)$ where $q_n \in \N$
satisfies $\psi(\|aq_n\|^{-1}) \ge q_n$ and $m_n$ is the nearest integer to
$aq_n$. Note that $|\xi^{(n)}_\lef| \sim |a-1| q_n \to \infty$. Our goal is to use the specific knowledge of the sequence realising the
$\psi$-Liouvilleanness of $\Gamma^\dagger$ in order to choose a window making
the discrepancy particularly large; the proof is almost identical to the proof
of Theorem \ref{ThmLiouville} but in fact more direct because we can put our
hands on the explicit sequence.

We put $\Omega_\dwn$ again to be any set of bounded perimeter, and $\Omega_\lef
= [-r,r] \times \Upsilon_\lef$, where the interval is taken to be in the
direction of $e_\lef$. This time, we have that for every $n \in \N$,
\begin{equation}
    [\CF \chi_{\Omega_\lef}](\xi^{(n)}) = 2 r \vol(\Upsilon_\lef) \frac{\sin(2 \pi r (q_n  - m_n))}{2
    \pi r (q_n - m_n)}
\end{equation}
Let $f : (0,\infty) \to (0,\infty)$ be any
function
such that $\lim_{t \to \infty} f(t) = \infty$. Set $t_n = |q_n a_n - m_n|^{-1} =
|\xi^{(n)}_\dwn|^{-1}$, and assume $n$ is large enough so that the bound in
\eqref{eq:boundxidown} holds. By Proposition
\ref{AppMainFormula} we have just as in \eqref{DefPsi} that
\begin{equation}\label{DefPsibis}
    \begin{aligned}
        \Psi_{f}(t_n,r) := \frac{\mathrm{NV}_{t_n}(\Gamma,
    r\Omega_\lef)}{t_n^{2d_\dwn}\psi(t_n)^{- 2}f(t_n)^{-1}} &=
    \frac{\psi(t_n)^{2}f(t_n) }{\covol(\Gamma)^2}\,  \sum_{\xi \in
\Gamma^\dagger \setminus\{(0,0)\}} |[\CF \chi_{\Omega_\dwn}](t_n \xi_\dwn)|^2
|[\CF \chi_{\Omega_\lef}](\xi_\lef)|^2  \\
&\ge \frac{\psi(t_n)^2 f(t_n) r^2 \vol(\Omega_\dwn)^2 \vol(\Upsilon_\lef)^2
}{\covol(\Gamma)^2} \frac{\sin^2(2 \pi r (q_n  - m_n))}{(2 \pi r(q_n  -
m_n))^2} \\
&\ge \frac{f(t_n) \vol(\Omega_\dwn)^2 \vol(\Upsilon_\lef)^2 }{4 \pi^2\covol(\Gamma)^2}
\sin^2(2 \pi r(q_n - m_n)) 
\end{aligned}
\end{equation}
where we have used that $|q_n - m_n| \le \psi(t_n)$. But again, \cite[Chapter 1,
Theorem 4.1]{Kui} tells us that we can find a Lebesgue conull set $A \subset
(0,\infty)$ for which the sequence $\set{r(q_n - m_n)}$ is equidistributed
modulo $1$ for all $r \in A$. Hence, for all $r \in A$ we can find a further
subsequence such that $r(q_{n_k} - m_{n_k}) - 1/4 \to 0$ (modulo $1$) as $k \to
\infty$. In particular, the last line in \eqref{DefPsibis} diverges to infinity
as $k \to \infty$ since $f$ is divergent.
\qed

\bibliographystyle{halpha}
\bibliography{QP_Patterns}

\end{document}